\documentclass[12pt]{article}
\usepackage[latin1]{inputenc}
\usepackage{amsfonts}
\usepackage{amsmath}
\usepackage{amssymb}
\usepackage{latexsym}

\usepackage[pagebackref]{hyperref}

\newtheorem{theorem}{Theorem}[section]
\newtheorem{lemma}[theorem]{Lemma}
\newtheorem{proposition}[theorem]{Proposition}

\newtheorem{defin}[theorem]{Definition}

\newenvironment{proof}{\noindent \textbf{Proof: }}{\hfill
$\Box$  \vspace{1ex}}
\newenvironment{definition}{\begin{defin}\em}{\end{defin}}
\newtheorem{defins}[theorem]{Definitions}

\newtheorem{exs}[theorem]{Examples}

\newtheorem{ex}[theorem]{Example}

\newtheorem{rem}[theorem]{Remark}

\newtheorem{rems}[theorem]{Remarks}

\newtheorem{corollary}[theorem]{Corollary}

\newcommand{\F}{\mathbb{F}}
\newcommand{\N}{\mathbb{N}}

\def\Bb {\mathcal{B}}

\def\bchi{\boldsymbol{\chi}}

\def \ord {\rm{ord}\,}
\def \Ord {\rm{Ord}\,}
\def \tA {\widetilde{A}}

\begin{document}

\begin{center}
{\Large\textbf{On the existence of pairs of primitive and normal 
elements over finite fields}}
\end{center}
\vspace{3ex}

\noindent\begin{center} 
\textsc{C\'{\i}cero Carvalho, Jo\~{a}o Paulo Guardieiro, Victor G.L. Neumann 
and Guilherme Tizziotti}\\ 
\vspace{1ex}
\small{Faculdade de Matem\'{a}tica, Universidade Federal de Uberl\^{a}ndia, 
Av. 
J. N. 
\'{A}vila  2121, 38.408-902 Uberl\^{a}ndia -- MG, Brazil }
\end{center}
%
%\title{On existence of some special pair of primitive and normal elements over 
%finite fields}
%
%\author{C. Carvalho, J.P. Guardieiro Sousa, V. Neumann and G. Tizziotti}
%
%%\maketitle
%
%\begin{abstract}
%
%\end{abstract}

\vspace{8ex}
\noindent
\textbf{Keywords:} Primitive element, normal element, normal basis, finite 
fields.\\
\noindent
\textbf{MSC:} 12E20, 11T23

\vspace{4ex}
\begin{small}
\begin{center}
\textbf{Abstract}
\end{center}
Let $\F_{q^n}$ be a finite field with $q^n$ elements, and let $m_1$ and $m_2$ 
be 
positive integers. Given polynomials  $f_1(x), f_2(x) \in \F_q[x]$ with 
$\deg(f_i(x)) \leq m_i$, for $i = 1, 2$, and such that the rational function 
$f_1(x)/f_2(x)$ 
satisfies  certain conditions which we define, we present a sufficient 
condition 
for 
the  existence of a primitive element $\alpha \in \F_{q^n}$, normal over 
$\F_q$, 
such that 
$f_1(\alpha)/f_2(\alpha)$ is also primitive.  
\end{small}

\section{Introduction}
Let $\F_q$ be a finite field with $q$ elements. 
An element $\alpha\in\mathbb{F}_q$ is called \textit{primitive} if it is a 
generator of the multiplicative cyclic group $\mathbb{F}_q^*$, so that there 
are $\phi(q-1)$ primitive elements in $\mathbb{F}_q$, where $\phi$ is the Euler 
totient 
function.  Let $n$ be a positive 
integer, an element $\beta \in \mathbb{F}_{q^n}$ is called \textit{normal over 
$\F_q$} if the set $\{\beta, \beta^q, \ldots, \beta^{q^{n - 1}} \}$ 
is a basis for 
$\mathbb{F}_{q^n}$ as an  $\mathbb{F}_q$-vector space. The primitive normal 
basis theorem states that for any $q$ and $n$ there exists an element in 
$\mathbb{F}_{q^n}$ which is simultaneously primitive and normal over $\F_q$. 

In their proof of this theorem (see \cite{CH}), Cohen and Huczynska developed a 
technique which 
has, since then, been adapted to treat other problems involving primitive and 
normal elements. For example, these same authors used a modified version of 
their technique to prove the strong normal basis theorem (see \cite{CH2}), 
which states that, 
except for a few pairs $(q,n)$, one can find an element $\alpha \in      
\mathbb{F}_{q^n}$  such that $\alpha$ and $\alpha^{-1}$ are primitive and 
normal over $\F_q$. Later, using the same line of reasoning, Kapetanakis (see 
\cite{kapeta})  proved 
that there exists 
an element $\alpha \in      
\mathbb{F}_{q^n}$  such that $\alpha$ and $(a \alpha + b)/(c \alpha + d)$, with 
$a,b,c,d \in  \F_q$, are 
primitive and 
normal over $\F_q$, except for a few combinations of $q$, $n$ and $a,b,c,d$. 
In 2017 Anju and Sharma, also following ideas from \cite{CH} and assuming that 
$q$ has characteristic two, proved that given polynomials $f(x), g(x) \in 
\F_{q^n}[x]$,  being $f(x)$ of degree at most 2 and $g(x)$ of degree at most 1, 
then there exists $\alpha \in \F_{q^n}$, primitive 
and normal over $\F_q$, such 
that $f(\alpha)/g(\alpha)$ is also primitive, except for a few combinations of 
$q = 2^k$, $m$, $f(x)$ and 
$g(x)$ (see \cite{mersenne}).  
More recently, Hazarika, Basnet and Cohen (\cite{HBC}) studied this problem 
working with a field of characteristic three and considering polynomials    
of degree at most two instead of rational functions. Hazarika and Basnet 
(\cite{HB})  also 
considered the related problem of finding pairs of elements $(\alpha, 
f(\alpha))$, both in $\F_{q^n}$, and both being primitive and normal over 
$\F_q$, where $f$ is a quotient of a polynomial of degree two by a polynomial 
of degree at most one, and $q$ has characteristic two.

In this paper we work with a finite field of any characteristic, and given 
polynomials 
$f_1(x), f_2(x) \in  \F_{q^n}$, of any degree,  we study, like Anju and Sharma, 
the 
existence of $\alpha \in \F_{q^n}$, primitive and normal over $\F_q$, such that
$f_1(\alpha)/f_2(\alpha)$ 
is also primitive. More specifically, given positive 
integers $m_1$ and $m_2$ we determine a set $\Upsilon_q(m_1, m_2)$ (see 
Definition \ref{upsilonset}) comprising 
certain rational functions $f(x)/g(x)$, where $\deg(f_i(x)) \leq m_i$, with $i 
= 1,2$,  and we determine conditions which assure, for each $f(x)/g(x) \in 
\Upsilon_q(m_1, m_2)$,  the existence of an element   
$\alpha \in \F_{q^n}$, primitive and normal over $\F_q$, such 
that $f(\alpha)/g(\alpha)$ is also primitive (see Corollary \ref{mainresult}).

In the following section we list the definitions and results which will be used 
in the proof of the main result, which is the content of Section 3. In Section 
4 we 
present  numerical examples  illustrating our main result.

\section{Preliminaries}

%\subsection{Primitive and normal elements} An element $\alpha\in\mathbb{F}_q$ 
%is called \textit{primitive} if $\alpha$ is a generator of the multiplicative 
%group
%$\mathbb{F}_q^*$, or equivalently, if the multiplicative order of $\alpha$ is 
%$q-1$. A pair $(\alpha,\beta) \in \mathbb{F}_q^2$ is a \textit{primitive pair} 
%in $\mathbb{F}_q$ if $\alpha$ and $\beta$ are primitive elements. It is clear 
%that $(\alpha,\beta) \in \mathbb{F}_q^2$ is a primitive pair if and only if 
%$(\alpha,\beta^{-1}) \in \mathbb{F}_q^2$ is a primitive pair. An element 
%$\alpha\in \F_{q^n}$ is called a \textit{normal element} of $\F_{q^n}$ over 
%$\F_q$ if
%$\{ \alpha, \alpha^q, \ldots , \alpha^{q^{n-1}}\}$ is a basis of $\F_{q^n}$ 
%over $\F_q$. In this case
%a basis of this form is called  a \textit{normal basis}.

Throughout this paper $p$ is a prime, $k$ is a positive integer, 
$\mathbb{F}_q$ will denote a finite field with $q=p^k$ elements and we denote
by $\mathbb{N}$ the set of positive integers.
 We start by defining a set that will play an important  role in what 
follows.

\begin{definition}  \label{upsilonset}
For $q,m_1,m_2 \in \mathbb{N}$, with $q$ a prime power, define  
$\Upsilon_{q} (m_1,m_2)$ as the set of rational functions 
$\frac{f_1}{f_2} \in \mathbb{F}_q(x)$
such that:
\begin{enumerate}  
\item[i)] 
$\deg (f_1) \leq m_1$,   $\deg (f_2) \leq m_2$; 
\item[ii)]
$\gcd(f_1,f_2)=1$;  
\item[iii)]
 there exists $n \in \mathbb{N}$ and an irreducible monic polynomial $g \in 
\F_q[x] \setminus \{x\}$ such that $\gcd(n, q - 1)=1$, $g^n  \mid  f_1 f_2$ and 
$g^{n+1} \nmid f_1 f_2$. 
\end{enumerate}
%b) We define $\Bb_p(m_1,m_2)$ as the set of pairs $(k,n) \in \mathbb{N}^2$ 
%such 
%that 
%$\mathbb{F}_{q^n}$ contains an element 
%$\alpha$ normal over $\mathbb{F}_q$ with $\alpha$ and $f(\alpha)$ 
%primitive elements for all $ f \in \Upsilon_{q^n} (m_1,m_2)$.
\end{definition}
%
%
%i) We define  $\Lambda_q(f_1,f_2)$ as the set of pairs $(n,g) \in \mathbb{N} 
%\times (\mathbb{F}_q [x]
%\setminus \{x\})$ such that $\gcd(n,q-1) = 1$, $g$ is monic, irreducible, $g^n 
%\mid f_1 f_2$ and  $g^{n+1} \nmid f_1 f_2$. \\
%iii)We define $\Gamma_p(m_1,m_2)$ as the set of positive integers $k$ such that
%$\mathbb{F}_{p^k}$  contains an element $\alpha$ with $(\alpha, f(\alpha))$  a 
%primitive pair  for all $f\in \Upsilon_{p^k} (m_1,m_2)$.\\

%$$
%\Bb_p(m_1,m_2)=
%\left\{
%(k,n) \mid
%\begin{array}{l}
%\exists \alpha \in \F_{q^n} \text{ such that }
%\alpha \text{ is normal over } \F_q  \\
%\text{and }
%\left(\alpha , f(\alpha) \right)
%\text{ is a primitive pair}  \\
%\text{for all } f \in \Upsilon_{q^n} (m_1,m_2), \text{ where } q=p^k
%\end{array}
%\right\}.
%$$
%\end{definition}

%
%In \cite{mersenne}, the authors studied a set similar to $\Bb_2(2,2)$.

\begin{definition}
Let  $s$ be a divisor of $q-1$, an element $\alpha\in\mathbb{F}_q^*$
is called \textit{$s$-free} if, for any $d \in \mathbb{N}$ such that $d\, |\, s$ and $d \neq 1$, there is no $\beta\in\mathbb{F}_q$ satisfying $\beta^d = \alpha$.
\end{definition}

%
%
%Thus, if $\alpha \in\mathbb{F}_q^*$ is $s$-free then  $\alpha$ cannot be a 
%$d$-th power of an element, where $d\neq1$ and $d|s$. Next, we present some 
%easy remarks that we will use in what follows.
%
%
%\begin{remark} \label{obs s_free}
%Let $\alpha\in\mathbb{F}_q^*$, then:
%
%\begin{enumerate}
%
%\item $\alpha$ is primitive if and only if $\alpha$ is $(q-1)$-free;
%
%\item if $\alpha$ is $s$-free for some integer $s$ then $\alpha$ is $e$-free 
%for any $e \, | \, s$;
%
%\item if $\alpha$ is $s_1$-free and $s_2$-free, then $\alpha$ is 
%$\mathrm{lcm}(s_1,s_2)$-free;
%
%\item let $p_1, \ldots, p_n$ be primes, then $\alpha$ is $(p_1 .\cdots. 
%p_n)$-free if and only if $\alpha$ is  $(p_{1}^{\alpha_1} .\cdots. 
%p_{n}^{\alpha_n})$-free, where $\alpha_i > 0$ for all $i$.
%
%
%\end{enumerate}
%
%\end{remark}
%

%\subsection{Characters} 

For $\beta \in \mathbb{F}_{q^n}$ and $\displaystyle 
f(x) = \sum_{i=1}^{t} f_i x^i\in \F_q[x]$, we define an action of 
$\mathbb{F}_q[x]$ over $\mathbb{F}_{q^n}$ by $\displaystyle f \circ \beta = 
\sum_{i=0}^{t} f_i \beta^{q^i}$. 
Through this action $\mathbb{F}_{q^n}$ may be viewed as an $\mathbb{F}_q[x]$-module, and the
annihilator of $\beta$ is an ideal of  $\mathbb{F}_q[x]$.
 The unique monic generator  $g$ of this ideal is the order of $\beta$, denoted by $\mbox{Ord}[\beta]$.
 Observe that clearly $g$ is a factor of $x^n - 1$.
One may prove that if $\mbox{Ord}[\beta]$ is $g$, then $\beta = h \circ \lambda$ for some $\lambda \in \mathbb{F}_{q^n}$, where $h=\frac{x^n-1}{g}$. 

Similarly  to the concept of $s$-free 
element,  we 
can define what is to be a $g$-free element for any $g \in 
\mathbb{F}_q[x]$ that divides $x^n-1$.

\begin{definition} \label{g-free}
Let $g \in \mathbb{F}_q[x]$ be such that $g|x^n-1$. An element $\alpha \in 
\mathbb{F}_{q^n}$ is said to be {\em $g$-free} if for any $h \in 
\mathbb{F}_q[x]$ such that $h|g$ and $\lambda \in \mathbb{F}_{q^n}$ we have 
that   $\alpha = h \circ \lambda$ implies $h=1$.
\end{definition}

From 
\cite[Section 3]{CH}
we know that the characteristic function of the 
set of $s$-free elements, with $s \mid q^n - 1$, is given 
by
\begin{equation} \label{funcao caracteristica}
\rho_s(\alpha) = 
\theta(s)\sum_{d|s}\frac{\mu(d)}{\phi(d)}\sum_{\chi_d}\chi_d(\alpha),
\end{equation}
where $\alpha \in \F_{q^n}$, $\alpha \neq 0$, $\theta(s):=\frac{\phi(s)}{s}$, 
$\mu$ is the 
Moebius's function and 
$\chi_d$ runs through the set of $\phi(d)$ multiplicative characters of 
$\mathbb{F}_{q^n}^*$ of order 
$d$.

We endow the group of additive characters of 
$\mathbb{F}_{q^n}$ with an structure of  $\mathbb{F}_q[x]$-module by means of  
the 
operation which combines a polynomial $f$ and a character $\psi$ to produce the 
character $\psi \circ f$ defined by $\psi 
\circ f(\beta) = 
\psi(f \circ \beta)$ for all  $\beta \in \mathbb{F}_{q^n}$.
The \textit{$\mathbb{F}_q$-order of an additive character $\psi$}, denoted by 
$\Ord{(\psi)}$, is defined to be the 
unique monic polynomial $g \in 
\mathbb{F}_q [x]$ of least degree dividing 
$x^n-1$ such that $\chi \circ g$ is the trivial character in $\mathbb{F}_{q^n}$.
There are $\Phi(g)$ additive characters of 
$\mathbb{F}_q$-order $g$, where $\Phi(g):= \left|(\mathbb{F}_q[x] / g 
\mathbb{F}_q[x])^{\ast}\right|$ is the analogue Euler's function on 
$\mathbb{F}_q[x]$. 
%(see \cite{CH}). % pg 44
Also in \cite[Section 3]{CH} we find the expression for the characteristic  
function 
for the set of $g$-free elements $\alpha \in \mathbb{F}_{q^n}$.
	For any $g \in \F_{q}[x]$ such that $g \mid x^n -1$ 
	this characteristic function $\kappa_{g}$
	is given by
%(see \cite[Lemma 5]{carlitz})
	\begin{equation} \label{char-g-free}
	\kappa_g(\alpha) = \frac{\Phi(g)}{N(g)}
	\sum_{h \mid g} \frac{\mu'(h)}{\Phi(h)}
	\sum_{\Ord (\psi)=h} \psi(\alpha),
	\end{equation}
	where
%	$\Phi(g) = \left|   \left(  \F_{q}[x] / g\F_{q}[x] \right)^* \right|$
%	is the analogue of Euler's function on $\F_{q}[x]$,
	$N(g) = \left|   \left(  \F_{q}[x] / g\F_{q}[x] \right) 
	\right|=q^{\deg(g)}$,
	the last sum runs over all additive characters $\psi$ of $\mathbb{F}_q^n$ 
	which have   
	$\F_{q}$-order $h$, and $\mu'$ is the M\"obius function on $\F_{q}[x]$
$$
\mu'(h) =
\left\{
\begin{array}{ll}
(-1)^s & {\textrm if} \,\, h \mbox{ is a product of } s \mbox{ disctinct 
monic 
irredutible polynomials;} \\
0 & \rm{ otherwise. }
\end{array}
\right.
$$

The next result is a combination of \cite[Theorem 5.5]{Fu} and a special case 
of \cite[Theorem 5.6]{Fu}, which we will need in what follows.

\begin{lemma} \label{lema cota}
Let $v(x),u(x) \in \F_{q^n}(x)$ be rational functions. Write
	$v(x)=\prod_{j=1}^k s_j(x)^{n_j}$, where
	$s_j(x) \in \F_{q^n}[x]$ are irreducible polynomials, pairwise 
	non-associated,  and
	$n_j$ are non-zero integers. Let $D_1=\sum_{j=1}^k \deg (s_j)$,
	let $D_2=\max (\deg (u), 0)$, let $D_3$ be the degree of the denominator of 
	$u(x)$
	and let $D_4$ be the sum of degrees of those irreducible polynomials 
	dividing the denominator of $u$,
	but distinct from $s_j(x)$ ($j=1,\ldots, k$). Let $\chi$ and $\psi$ be, 
	respectively, a multiplicative character and 
	a non-trivial additive character of $\F_{q^n}$.\\
a) Assume that $v(x)$ is not of the form $r(x)^{ord(\chi)}$
	in $\mathbb{F}(x)$, where $\mathbb{F}$ is an algebraic closure of 
	$\mathbb{F}_{q^n}$. Then 
$$
	\displaystyle \left|
	\sum_{\substack{\alpha \in \mathbb{F}_{q^n}  \\ v(\alpha)\neq 0 , 
	v(\alpha)\neq
			\infty}} \chi(v(\alpha)) \right|
	\leq
	(D_1-1 ) q^{\frac{n}{2}}.
$$\\
b) Assume that
	$u(x)$ is not of the form $r(x)^{q^n}-r(x)$ in $\F(x)$, where $\mathbb{F}$ 
	is an algebraic closure of $\mathbb{F}_{q^n}$. Then
$$
	\displaystyle \left|
	\sum_{\substack{\alpha \in \mathbb{F}_{q^n}  \\ v(\alpha)\neq 0, 
	v(\alpha)\neq \infty , \\
	u(\alpha)\neq \infty}} \chi(v(\alpha)) \psi(u(\alpha)) \right|
	\leq
	\left( D_1 + D_2 + D_3 + D_4 - 1  \right) q^{\frac{n}{2}}.
$$	
\end{lemma}

\section{Main results}\label{section-main}

Let $m_1$ and $m_2$ be positive integers, we want to determine 
conditions on $q$ and $n$ such that
for each $f \in \Upsilon_{q^n} (m_1,m_2)$
there exists $\alpha \in \mathbb{F}_{q^n}$, primitive and normal
over $\mathbb{F}_q$, such that $f(\alpha) \in \mathbb{F}_{q^n}$  is also a 
primitive element.	
For this we will need the following concept.
\begin{definition} 	\label{upsilon}
Let $q=p^k$, let
$e_1$ and $e_2$ be divisors of $q^n-1$ and let $g$ be a divisor of $x^n-1$.
Given $f \in \Upsilon_{q^n}(m_1,m_2)$
we will denote by $N_{f}(e_1,e_2,g)$ the number of $\alpha \in \F_{q^n}$ such 
that
$\alpha$ is $e_1$-free, $f(\alpha)$ is $e_2$-free and
$\alpha$ is $g$-free.
\end{definition}

It is easy to check that 
$\alpha \in \F_{q^n}$ is primitive if and only if $\alpha$ is $(q^n-1)$-free, 
and 
that $\beta \in \mathbb{F}_{q^n}$ is normal over $\mathbb{F}_q$ if and only 
if $\beta$ is $(x^n - 1)$-free. We want to find conditions which assure that
$N_{f}(q^n-1,q^n-1,x^n-1)>0$ for all
$f \in \Upsilon_{q^n} (m_1,m_2)$, yet in the next result we deal with a 
slightly more general situation. Before stating it, we observe that when $n = 
1$ or $n = 2$ then every primitive element in $\mathbb{F}_{q^n}$ is normal over 
$\mathbb{F}_q$, so we may ignore the ``normal''  requirement and  the problem 
was already solved in \cite{CSS}. Thus we assume from now on that $n \geq 
3$.

For $\ell \in \mathbb{N}$ we denote by $W(\ell)$ the number of 
distinct square-free  divisors of $\ell$, and
for a polynomial $g \in \F_q[x]$ we denote by
$W_q(g)$ the number
of monic square free factors of $g$ in $\F_{q}[x]$.

\begin{theorem}\label{principal}
	Let  $e_1$ and $e_2$ be divisors of   $q^n-1$, let 
	 $g \in \F_q[x]$ be a factor of $x^n -1$ and let $f \in \Upsilon_{q^n} 
	 (m_1,m_2)$. Then
\begin{eqnarray}
	\displaystyle N_{f}(e_1,e_2,g) & > &\frac{\phi(e_1)\phi(e_2) \Phi(g)}{e_1 
	e_2 
		N(g)}
	\left(
	q^n-(m_1+m_2+1)- \right. \nonumber \\
	& & \left. (m_1+m_2+1)q^{\frac{n}{2}}(W(e_1)W(e_2)W_q(g)-1)
	\right), \nonumber
\end{eqnarray}
	and a fortiori  
	if $q^{n/2} \geq (m_1 + m_2 + 1) W(e_1)W(e_2)W_q(g)$ then 
	$N_{f}(e_1,e_2,g)>0$.
\end{theorem}
\begin{proof}
Let $f = \dfrac{f_1}{f_2} \in \Upsilon_{q^n}(m_1,m_2)$
%, with
%	$\deg f_1 = \widetilde{m}_1 \leq m_1$ and $\deg f_2=\widetilde{m}_2 \leq 
%m_2$, and let %$e_1,e_2 \mid q^n-1$
%	and $g \in \F_q[x]$ with $g \mid x^n -1$.
%Define
and let
$$
S_{f}:=
\left\{ \alpha \in \F_{q^n} \mid f_1(\alpha)=0 \text{ or } f_2(\alpha)=0  
\right\} \cup \{ 0 \}.
$$

From the definition of $N_{f}(e_1,e_2,g)$ and equations \eqref{funcao 
caracteristica} and (\ref{char-g-free})	 we have
\begin{eqnarray}\label{eqNf}
N_{f}(e_1,e_2,g) & = &
\sum_{\alpha \in \F_{q^n} \backslash S_{f}}
\rho_{e_1}(\alpha) \rho_{e_2} (f(\alpha)) \kappa_g(\alpha)  \\
& = &  \nonumber
\frac{\phi(e_1)\phi(e_2) \Phi(g)}{e_1 e_2 N(g)}
\sum_{\substack{d_1\mid e_1 , d_2\mid e_2 \\ h\mid g}}
\frac{\mu(d_1)\mu(d_2) \mu'(h)}{\phi(d_1)\phi(d_2) \Phi(h)}
\sum_{\substack{\ord(\chi_1)=d_1 \\ \ord(\chi_2)=d_2 \\ \Ord(\psi)=h}}
\tilde{\bchi}_f (\chi_1,\chi_2,\psi), 
\end{eqnarray}
where
$$
\tilde{\bchi}_f (\chi_1,\chi_2,\psi)
=
\sum_{\alpha \in \F_{q^n} \backslash S_{f}}
\chi_1(\alpha) \chi_2(f(\alpha)) \psi(\alpha) .
$$

To find a bound for $N_f(e_1,e_2,g)$ we will
bound $| \tilde{\bchi}_f(\chi_{1},\chi_{2},\psi) |$,
and we consider five cases.
\begin{enumerate}
\item[(i)] We first consider the case where $\chi_1$, $\chi_2$ and
$\psi$ are trivial characters, so that
$$
\tilde{\bchi}_f(\chi_{1},\chi_{2},\psi) = \left|\mathbb{F}_{q^n} \setminus S_f 
\right|\geq
q^n - (m_1+m_2+1).
$$

\item[(ii)] Now we deal with the case where $\chi_1$ and $\chi_2$ are trivial 
multiplicative  characters, while $\psi$ is not a trivial additive character. 
It is well known that 
$\displaystyle \sum_{\alpha \in \mathbb{F}_{q^n}}\psi(\alpha)= 0$, so that
$$
|\tilde{\bchi}_f(\chi_{1},\chi_{2},\psi) | =
\left| \sum_{\alpha\in\F_{q^n}\backslash S_f} \psi (\alpha)
\right| =
\left| - \sum_{\alpha\in S_f} \psi (\alpha)
\right| \leq m_1 + m_2+1.
$$

\item[(iii)] We treat the case where $\chi_1$ is not a trivial character, while 
$\chi_2$ and $\psi$ are trivial characters. It is well known that 
$\sum_{\alpha \in \mathbb{F}_{q^n}^*}\chi_1(\alpha)= 0$, so we have 
\begin{eqnarray*}
|\tilde{\bchi}_f(\chi_{1},\chi_{2},\psi) | & =&
\left| \sum_{\alpha \in \mathbb{F}_{q^n}^*}\chi_1(\alpha)
-
\sum_{\alpha\in\mathbb{F}_{q^n}\setminus S_f}\chi_1(\alpha)
\right| =
\left| \sum_{\alpha \in S_f\setminus\{0\}}\chi_1(\alpha) \right| \\
& \leq &(m_1+m_2) < (m_1+m_2)q^{\frac{n}{2}}.
\nonumber
\end{eqnarray*}
\end{enumerate}

Before proceeding to treat the cases where we assume at most one trivial 
character, we will rewrite the expression for 
$\tilde{\bchi}_f(\chi_{1},\chi_{2},\psi)$. 

Let $\chi_1$ and $\chi_2$ be multiplicative characters of orders $d_1$ and
$d_2$, respectively, where $d_1 \, |\, e_1$ and $d_2 \, | \, e_2$
and let $\psi$ be an additive character of $\F_q$-order $h$.
Let $i \in \{1, 2\}$,
it is
well-known (see e.g. \cite[Thm.\ 5.8]{LN}) that there exists a character
$\chi$  of order $q^n -1$ and and integer
$n_i \in\{0,1,...,q-2\}$ such that $\chi_{i}(\alpha)=\chi(\alpha^{n_i})$
for all $\alpha \in \mathbb{F}_{q^n}^*$, and observe that
$n_i=0$ if and only if $\chi_i$ is a trivial character. Hence,
\begin{eqnarray*}
\tilde{\bchi}_f(\chi_{1},\chi_{2},\psi) & = &
\sum_{\alpha\in\F_{q^n}\backslash S_f}
\chi(\alpha^{n_1}f_1(\alpha)^{n_2}f_2(\alpha)^{-n_2})\psi(\alpha) \\
& = &
\sum_{\alpha\in\F_{q^n}\backslash S_f}\chi(v(\alpha))\psi(\alpha), 
\end{eqnarray*}
where $v(x)=x^{n_1}f_1(x)^{n_2}f_2(x)^{-n_2}$.

\begin{enumerate}
\item[(iv)] Now we assume that $\psi$ is a trivial additive character, while 
$\chi_2$ is not a trivial multiplicative character, so that $n_2 \neq 0$, and 
we make no assumptions 
on $\chi_1$.
To bound $\tilde{\bchi}_f(\chi_{1},\chi_{2},\psi) $
we want to use Lemma \ref{lema cota} (a), and we
start by showing that indeed we can use it. So we suppose by means of absurd
that $v(x) = \left( \frac{v_1(x)}{v_2(x)} \right)^{q^n-1}$
for some
$v_1(x),v_2(x)\in\mathbb{F}[x]$, with 
%$\textrm{deg}(v_1)=r_1$,
%$\textrm{deg}(v_2)=r_2$, and
$\mbox{gcd}(v_1,v_2)=1$, then
\begin{equation*}
x^{n_1}f_1(x)^{n_2}v_2(x)^{q^n-1}=f_2(x)^{n_2}v_1^{q^n-1}(x).
\end{equation*}
Since $\frac{f_1(x)}{f_2(x)} \in \Upsilon_{q^n}
(m_1,m_2)$, there exists an
irreducible monic polynomial $t(x) \in \mathbb{F}_{q^n}[x]$, $t(x)\neq x$ and a
positive integer $a$
with $\mbox{gdc}\, (a,q^n-1)=1$ such that
$t(x)^a$ is the largest power of $t(x)$ which appears in the factorization of 
either $f_1(x)$ or $f_2(x)$. Let's
suppose that $t(x)^a$ appears in the factorization of $f_2(x)$, and let
$\tilde{t}(x)$ be an irreducible factor of $t(x)$ in $\mathbb{F}[x]$. Clearly
$\tilde{t}(x)$ has degree one, $\tilde{t}(x) \neq x$ and since 
$\mathbb{F}_{q^n}$
is a
perfect field we
know that $\tilde{t}(x)$ appears with multiplicity one in the factorization of
$t(x)$ in $\mathbb{F}[x]$.
Since $f_1(x)$ and $f_2(x)$ are coprime in $\mathbb{F}_{q^n}[x]$ they are also
coprime in  $\mathbb{F}[x]$ so
$\tilde{t}(x)^{a n_2}$ is the largest power of $\tilde{t}(x)$ which appears in 
the factorization of $v_2(x)^{q^n-1}$. From
this one may conclude that
$q^n-1 \, | \, a n_2$, and from $\gcd(a,q^n-1)=1$ we get $q^n-1 \, | \, n_2$, a
contradiction. So we must have that
$t(x)^a$ appears in the factorization of $f_1(x)$, and reasoning as above again
we
conclude again that $q^n-1\, | \, n_2$, which is impossible. Thus, if 
$n_2 \neq 0$ we get that $v(x)$ is not of the form $\left( 
\frac{v_1(x)}{v_2(x)} \right)^{q^n-1}$
in $\mathbb{F}(x)$.

Let $T_v$ be the set of $\beta \in \mathbb{F}_{q^n}$ such that $v(\beta) = 0$ or
$v(\beta)$ is not defined. If $0 \in T_v$ then $T_v = S_f$ and from Lemma
\ref{lema cota} we have
$$
\left| \tilde{\bchi}_f(\chi_{1},\chi_{2},\psi) \right|
=
\left| \sum_{\alpha\in\F_{q^n}\setminus S_f}\chi(v(\alpha))
\right| =
\left| \sum_{\alpha\in\F_{q^n}\setminus T_v}\chi(v(\alpha))
\right|
\leq (m_1 + m_2) q^{\frac{n}{2}}.
$$
If $0 \notin T_v$ then %$n_1 = 0$ and
$$
\left| \tilde{\bchi}_f(\chi_{1},\chi_{2},\psi)\right|
=
\left| \sum_{\alpha\in\F_{q^n}\setminus S_f}\chi(v(\alpha))
\right| =
\left| \sum_{\alpha\in\F_{q^n}\setminus T_v}\chi(v(\alpha)) - \chi(v(0))
\right|,
$$
so
$\left| \tilde{\bchi}_f(\chi_{1},\chi_{2},\psi)\right|
\leq (m_1 + m_2 - 1) q^{\frac{n}{2}} + 1$ and
anyway we get $| \tilde{\bchi}_f(\chi_{1},\chi_{2},\psi) | \leq (m_1 + m_2)
q^{\frac{n}{2}}$.

%Besides that $v(x)\neq \left( \frac{v_1(x)}{v_2(x)} \right)^{q^n-1}$ and $n_2 
%\neq 1$. From Lemma \ref{cotaadmult} it follows that
%$$
%\left| \tilde{\bchi}_f(\chi_{1},\chi_{2},\psi) \right|
%=
%\left| \sum_{\alpha\in\F_{q^n}\backslash S_v}\chi(v(\alpha))
%-
%\sum_{\alpha\in S_f\backslash S_v}\chi(v(\alpha)) \right|
%\leq M q^{\frac{n}{2}} + |S_f \backslash S_v|,
%$$
%where $M$ is the constant of Lemma \ref{cotaadmult}.
%
%Now, if $d_1=1$, then
%$n_1=0$ and so $v(\alpha) =f_1(\alpha )^{n_2}f_2(\alpha )^{-n_2}$.
%In this case $M=m_1 + m_2 - 1$ and
%$|S_f \backslash S_v| \leq 1$.
%If $d_1\neq 1$, then
%$v(\alpha) =\alpha^{n_1 }f_1(\alpha )^{n_2}f_2(\alpha )^{-n_2}$.
%In this case $M=m_1 + m_2$ and
%$S_f = S_v$. In both cases we get
%$$
%\left| \tilde{\bchi}_f(\chi_{1},\chi_{2},\psi) \right|
%<
%(m_1+m_2+1) q^{\frac{n}{2}}.
%$$

\item[(v)] 
Lastly we consider the case where $\psi$ is not a trivial character, and either 
$\chi_1$ or $\chi_2$ is not a trivial character, so that  $d_1\neq 1$ or 
$d_2\neq 1$.
Obviously 
$x$ is not of the form $r(x)^{q^n}-r(x)$ in $\F(x)$, so we may use
Lemma \ref{lema cota} (b).

As in the above case
let $T_v$ be the set of $\beta \in \mathbb{F}_{q^n}$ such that $v(\beta) = 0$ or
$v(\beta)$ is not defined. If $0 \in T_v$ then $T_v = S_f$ and from Lemma
\ref{lema cota} we have
$$
\left| \tilde{\bchi}_f(\chi_{1},\chi_{2},\psi) \right|
=
\left| \sum_{\alpha\in\F_{q^n}\setminus T_v}\chi(v(\alpha)) \psi(\alpha)
\right| 
\leq (m_1 + m_2+1) q^{\frac{n}{2}}.
$$
If $0 \notin T_v$ then %$n_1 = 0$ and
$$
\left| \tilde{\bchi}_f(\chi_{1},\chi_{2},\psi)\right|
 =
\left| \sum_{\alpha\in\F_{q^n}\setminus T_v}\chi(v(\alpha))\psi(\alpha) -
\chi(v(0)) \psi(\alpha)
\right| \leq (m_1 + m_2) q^{\frac{n}{2}} + 1 ,
$$
and
anyway we get $| \tilde{\bchi}_f(\chi_{1},\chi_{2},\psi) | \leq (m_1 + m_2 + 1)
q^{\frac{n}{2}}$.
\end{enumerate}

This finishes the analysis of the possibilities for the characters $\chi_1$, 
$\chi_2$ and $\psi$, and 
now we use the above estimates to bound $N_{f}(e_1,e_2,g)$. 
Let $\chi_0$ be the trivial multiplicative character and let
$\psi_0$ be the trivial additive character.
Write
$$
N_{f}(e_1,e_2,g)
=
\frac{\phi(e_1)\phi(e_2) \Phi(g)}{e_1 e_2 N(g)}
(S_1 + S_2 + S_3 + S_4 + S_5),
$$
where
$$S_1=\tilde{\bchi}_f (\chi_0,\chi_0,\psi_0),$$ % \geq q^n - (m_1+m_2+1)$,
$$
S_2=
\sum_{\substack{h\mid g \\ h \neq 1}}
\frac{\mu'(h)}{\Phi(h)}
\sum_{\substack{\Ord(\psi)=h}}
\tilde{\bchi}_f (\chi_0,\chi_0,\psi),
$$
$$
S_3=
\sum_{\substack{d_1\mid e_1  \\
		                   d_1 \neq 1}}
\frac{\mu(d_1)}{\phi(d_1)}
\sum_{\ord(\chi_1)=d_1}
\tilde{\bchi}_f (\chi_1,\chi_0,\psi_0),
$$
$$
S_4=
\sum_{\substack{d_1\mid e_1 , d_2\mid e_2 \\ 
		d_2 \neq 1}}
\frac{\mu(d_1)\mu(d_2)}{\phi(d_1)\phi(d_2)}
\sum_{\substack{\ord(\chi_1)=d_1 \\ \ord(\chi_2)=d_2}}
\tilde{\bchi}_f (\chi_1,\chi_2,\psi_0)
$$
and
$$
S_5=
\sum_{\substack{d_1\mid e_1 , d_2\mid e_2 \\ 
		   d_1 \neq 1 \text{ or } d_2 \neq 1 \\
	        1 \neq h\mid g}}
\frac{\mu(d_1)\mu(d_2) \mu'(h)}{\phi(d_1)\phi(d_2) \Phi(h)}
\sum_{\substack{\ord(\chi_1)=d_1 \\ \ord(\chi_2)=d_2 \\ \Ord(\psi)=h}}
\tilde{\bchi}_f (\chi_1,\chi_2,\psi).
$$
From what we did above and using that
there are $\phi(d_1)$ multiplicative characters of order $d_1$, $\phi(d_2)$
multiplicative characters of order $d_2$
and $\Phi(h)$ additive characters of $\F_q$-order $h$
 we get
\begin{eqnarray}
|S_2+S_3+S_4+S_5| & < &
(m_1+m_2+1)q^{\frac{n}{2}}
\left(
\sum_{\substack{d_1 \mid e_1, d_2\mid e_2, h \mid g \\ (d_1,d_2,h) \neq 
(1,1,1)}}
|\mu(d_1)| |\mu(d_2)| |\mu'(h)|
\right) \nonumber \\
& = & 
(m_1+m_2+1)q^{\frac{n}{2}}
\left(
W(e_1)W(e_2)W(g) - 1
\right). \nonumber
\end{eqnarray}
Therefore, we conclude that

\begin{equation} 
\begin{array}{lll}     \label{eqpag10}
\displaystyle N_{f}(e_1,e_2,g) & > &\dfrac{\phi(e_1)\phi(e_2) \Phi(g)}{e_1 e_2 
N(g)}
\left(
q^n-(m_1+m_2+1)- \right.  \\
& & \left. (m_1+m_2+1)q^{\frac{n}{2}}(W(e_1)W(e_2)W_q(g)-1)
\right). 
\end{array}
\end{equation}

Thus, if
\begin{eqnarray*}
q^n & \geq & (m_1+m_2+1)q^{\frac{n}{2}}(W(e_1)W(e_2)W_q(g))
 \\
      & > & (m_1+m_2+1)+(m_1+m_2+1)q^{\frac{n}{2}}(W(e_1)W(e_2)W_q(g)-1) ,
\end{eqnarray*}
then
$N_f(e_1,e_2,g)>0$.
\end{proof}

\begin{corollary}\label{mainresult}
If
$q^{\frac{n}{2}} \geq (m_1 + m_2+1) W(q^n-1)^2W_q(x^n-1)$
then 
for each $ f \in \Upsilon_{q^n} (m_1,m_2)$
there exists $\alpha \in \mathbb{F}_{q^n}$, primitive and normal
over $\mathbb{F}_q$, such that $f(\alpha) \in \mathbb{F}_{q^n}$  is also a 
primitive element.	
\end{corollary}

The next result presents a inequality similar to others which have appeared in  
previous works about primitive or normal elements.

\begin{lemma}\label{lemmasieve}
Let $\ell$ be a divisor of $q^n-1$ and let $\{p_1,...,p_r\}$ be the set of 
	all 
	primes which divide $q^n-1$,
	but do not divide $\ell$.
	Also let $g \in \F_q[x]$ be a divisor of $x^n -1$ and $\{P_1,...,P_s\} 
	\subset \F_q[x]$
	be the set of all monic irreducible polynomials which
	divide $x^n -1$, but do not divide $g$. Then
\begin{eqnarray}
N_f(q^n-1,q^n-1,x^n-1) & \geq &
\sum_{i=1}^{r}N_f(p_i \ell,\ell,g)
+
\sum_{i=1}^{r}N_f(\ell,p_i \ell,g) \nonumber \\
& &
+\sum_{i=1}^{s}N_f( \ell,\ell, P_i g)
-
(2r+s-1)N_f(\ell,\ell,g). \label{sieve}
\end{eqnarray}
\end{lemma}

\begin{proof}
The left side of \eqref{sieve} counts every $\alpha \in \F_{q^n}$ for which $\alpha$ is primitive and normal, and $f(\alpha)$ is primitive.
Observe that if $\alpha$ is primitive and normal then 
$\alpha$ is $\ell$-free, $g$-free, $p_i \ell$-free for all $i \in \{ 1,\ldots , 
r\}$ and $P_i g$-free for all $i \in \{ 1,\ldots , s\}$, also if 
$f(\alpha)$ is primitive then $f(\alpha)$ is $\ell$-free and 
$p_i \ell$-free for all $i \in \{ 1,\ldots , r\}$. This shows that 
if $\alpha$ is primitive and normal and $f(\alpha)$ is primitive then $\alpha$ 
is counted $2r+s - (2r+s-1)=1$
times on the right side of \eqref{sieve}.
For any other $\alpha \in \F_{q^n}$, we have that either $\alpha$ or 
$f(\alpha)$ is not $p_i \ell$-free for some $i \in \{ 1,\ldots , r\}$, or 
$\alpha$ is not $P_i g$-free for some $i \in \{ 1,\ldots , s\}$, so $\alpha$ 
will not be counted in at least one of the first three sums of \eqref{sieve}.
\end{proof}

The next result will be useful in the calculations which appear in the next 
section.

\begin{lemma} \label{divisores}
Let $\ell$ be a divisor of $q^n-1$ and let $\{p_1,...,p_r\}$ be the set of all 
primes which divide $q^n-1$,
but do not divide $\ell$.
Also let $g \in \F_q[x]$ be a divisor of $x^n -1$ and $\{P_1,...,P_s\} \subset 
\F_q[x]$
be the set of all monic irreducible polynomials which
divide $x^n -1$, but do not divide $g$.
Suppose that
$$
\delta=1-2\sum_{i=1}^{r}\frac{1}{p_i} - \sum_{i=1}^{s}\frac{1}{q^{\deg P_i}} >0
$$
and let  $\Delta=\frac{2r+s-1}{\delta}+2$. If 
$q^{\frac{n}{2}} \geq 
(m_1 + m_2+1)
W(\ell)^2W_q(g) \Delta$, then 
for each $ f \in \Upsilon_{q^n} (m_1,m_2)$
there exists $\alpha \in \mathbb{F}_{q^n}$, primitive and normal
over $\mathbb{F}_q$, such that $f(\alpha) \in \mathbb{F}_{q^n}$  is also a 
primitive element.	
\end{lemma}

\begin{proof}
%Taking $q_1=q_2=q^n-1$, $F_1 = x^n-1$ and $F_2 =1$ in the quadruple divisor 
%$\mathbf{k}$ of the Sieving Inequality given in \cite[Proposition 5.2]{kapeta0} 
From \ref{lemmasieve}
we have that
\begin{eqnarray} \label{sieveineq}
N_f(q^n-1,q^n-1,x^n-1) & \geq &
\sum_{i=1}^{r}N_f(p_i \ell,\ell,g)
+
\sum_{i=1}^{r}N_f(\ell,p_i \ell,g)  \\
& &
+\sum_{j=1}^{s}N_f( \ell,\ell, P_j g)
-
(2r+s-1)N_f(\ell,\ell,g). \nonumber
\end{eqnarray}

Since $\theta(p_i) = \dfrac{\phi(p_i)}{p_i} = 1 - \dfrac{1}{p_i}$, for all $i = 
1, \ldots, r$, and $\dfrac{\Phi(P_j)}{N(P_j)} = 1 - \dfrac{1}{q^{\deg P_j}}$, 
for all $j=1,\ldots , s$, we can rewrite the right side of the inequality above 
obtaining

\begin{equation*}
\begin{split}
N_f(q-1,q-1,x^n-1)\geq & \sum_{i=1}^{r}(N_f(p_i \ell,\ell,g) - 
\theta(p_i)N_f(\ell,\ell,g)) \\& +\sum_{i=1}^{r}(N_f(\ell,p_i\ell,g) - 
\theta(p_i)N_f(\ell,\ell,g))  + \\ & \sum_{j=1}^{s} (N_{f}(\ell, \ell , P_j g) 
- 
\dfrac{\Phi(P_j)}{N(P_j)} N_f (\ell,\ell,g))  + \delta N_f(\ell,\ell,g).
\end{split}
\end{equation*}

From  $\theta(\ell p_i) = \theta(p_i) \theta(\ell)$ and \eqref{eqNf} 
we get
\begin{equation}
N_{f}(p_i \ell,\ell,g) = \frac{\phi(p_i \ell)\phi(\ell) \Phi(g)}{p_i \ell \ell 
N(g)}
\sum_{\substack{d_1\mid p_i \ell , d_2\mid \ell \\ h\mid g}}
\frac{\mu(d_1)\mu(d_2) \mu'(h)}{\phi(d_1)\phi(d_2) \Phi(h)}
\sum_{\substack{\ord(\chi_1)=d_1 \\ \ord(\chi_2)=d_2 \\ \Ord(\psi)=h}}
\tilde{\bchi}_f (\chi_1,\chi_2,\psi)
\end{equation}
for all $i = 1, \ldots, r$. Also for all $i = 1, \ldots, r$
we split the set of $d_1$'s which divide $p_i \ell$ into two sets: the first 
one contains those which do not have $p_i$ as a factor, while the second
one contains those which are a multiple of $p_i$. This will split the first 
summation into two sums, so that
\begin{eqnarray*}
N_{f}(p_i \ell,\ell,g) & = & \\
& & \frac{\phi(p_i \ell)\phi(\ell) \Phi(g)}{p_i \ell \ell N(g)}
\sum_{\substack{d_1\mid \ell , d_2\mid \ell \\ h\mid g}}
\frac{\mu(d_1)\mu(d_2) \mu'(h)}{\phi(d_1)\phi(d_2) \Phi(h)}
\sum_{\substack{\ord(\chi_1)=d_1 \\ \ord(\chi_2)=d_2 \\ \Ord(\psi)=h}}
\tilde{\bchi}_f (\chi_1,\chi_2,\psi) \\
& + & \frac{\phi(p_i \ell)\phi(\ell) \Phi(g)}{p_i \ell \ell N(g)}
\sum_{\substack{p_i | d_1 , d_1\mid p_i \ell , d_2\mid \ell \\ h\mid g}}
\frac{\mu(d_1)\mu(d_2) \mu'(h)}{\phi(d_1)\phi(d_2) \Phi(h)}
\sum_{\substack{\ord(\chi_1)=d_1 \\ \ord(\chi_2)=d_2 \\ \Ord(\psi)=h}}
\tilde{\bchi}_f (\chi_1,\chi_2,\psi)
\end{eqnarray*}
and from the expression for $N_f(\ell,\ell,g)$ (see \eqref{eqNf}) we 
get
\begin{equation*}
\begin{split}
N_f(p_i \ell,\ell,g) -& \theta(p_i) N_f(\ell, \ell,g) \\& = \frac{\phi(p_i 
\ell)\phi(\ell) \Phi(g)}{p_i \ell \ell N(g)}
\sum_{\substack{p_i | d_1 , d_1\mid p_i \ell , d_2\mid \ell \\ h\mid g}}
\frac{\mu(d_1)\mu(d_2) \mu'(h)}{\phi(d_1)\phi(d_2) \Phi(h)}
\sum_{\substack{\ord(\chi_1)=d_1 \\ \ord(\chi_2)=d_2 \\ \Ord(\psi)=h}}
\tilde{\bchi}_f (\chi_1,\chi_2,\psi).
\end{split}
\end{equation*}
	
From (iv) and (v) in the proof of Theorem \ref{principal}, and from
$$
\displaystyle \sum_{\substack{p_i | d_1 , d_1\mid p_i \ell , d_2\mid \ell \\ 
h\mid g}}
|\mu(d_1)| |\mu(d_2)| |\mu'(h)| = W(l)^2 W_{q}(g) 
$$
we conclude that
\begin{equation} \label{N1}
\left| N_f(p_i \ell,\ell,g) - \theta(p_i) N_f(\ell, \ell,g) \right| \leq
(m_1 + m_2 +1) \theta(p_i) \theta(\ell)^2 \dfrac{\Phi(g)}{N(g)}  W(\ell)^2 
W_{q}(g) q^{\frac{n}{2}}.
\end{equation}

In a similar way, we get
\begin{equation} \label{N2}
\left| N_f( \ell,p_i \ell,g) - \theta(p_i) N_f(\ell, \ell,g) \right| \leq
(m_1 + m_2 +1) \theta(p_i) \theta(\ell)^2  \dfrac{\Phi(g)}{N(g)} W(\ell)^2 
W_{q}(g) q^{\frac{n}{2}},
\end{equation}
for all $i=1,\ldots,r$.

Again from  (\ref{eqNf}) and using that $\Phi(P_j g) = \Phi(P_j) \Phi(g)$ 
and $N(P_j g) = N(P_j)N(g)$ we get
\begin{equation*}
N_{f}(\ell,\ell,P_jg) = \frac{\phi(\ell)\phi(\ell) \Phi(P_j) \Phi(g)}{\ell \ell 
N(P_j)N(g)}
\sum_{\substack{d_1\mid \ell , d_2\mid \ell \\ h\mid P_j g}}
\frac{\mu(d_1)\mu(d_2) \mu'(h)}{\phi(d_1)\phi(d_2) \Phi(h)}
\sum_{\substack{\ord(\chi_1)=d_1 \\ \ord(\chi_2)=d_2 \\ \Ord(\psi)=h}}
\tilde{\bchi}_f (\chi_1,\chi_2,\psi).
\end{equation*}

Splitting the set of $h$'s which divide $P_j g$ into two sets, namely, the 
first 
one contains those which do not have $P_j$ as a factor, while the second
one contains those which are a multiple of $P_j$, we have
\begin{eqnarray*}
N_{f}(\ell,\ell,P_j g) & = & \\
& & \frac{\phi(\ell)\phi(\ell) \Phi(P_j) \Phi(g)}{\ell \ell N(P_j)N(g)}
\sum_{\substack{d_1\mid \ell , d_2\mid \ell \\ h\mid g}}
\frac{\mu(d_1)\mu(d_2) \mu'(h)}{\phi(d_1)\phi(d_2) \Phi(h)}
\sum_{\substack{\ord(\chi_1)=d_1 \\ \ord(\chi_2)=d_2 \\ \Ord(\psi)=h}}
\tilde{\bchi}_f (\chi_1,\chi_2,\psi) \\
& + & \frac{\phi(\ell)\phi(\ell) \Phi(P_j) \Phi(g)}{\ell \ell N(P_j)N(g)}
\sum_{\substack{ d_1\mid \ell , d_2\mid \ell \\ P_j \mid h , h\mid P_jg}}
\frac{\mu(d_1)\mu(d_2) \mu'(h)}{\phi(d_1)\phi(d_2) \Phi(h)}
\sum_{\substack{\ord(\chi_1)=d_1 \\ \ord(\chi_2)=d_2 \\ \Ord(\psi)=h}}
\tilde{\bchi}_f (\chi_1,\chi_2,\psi).
\end{eqnarray*}

From the expression for $N_f(\ell,\ell,g)$ (see \eqref{eqNf}) we get
\begin{equation*}
\begin{split}
N_{f}(\ell, \ell , P_j g) -& \Phi(P_j) N_f (\ell,\ell,g) \\ &= 
\frac{\phi(\ell)\phi(\ell) \Phi(P_j) \Phi(g)}{\ell \ell N(P_j)N(g)}
\sum_{\substack{ d_1\mid \ell , d_2\mid \ell \\ P_j \mid h , h\mid P_jg}}
\frac{\mu(d_1)\mu(d_2) \mu'(h)}{\phi(d_1)\phi(d_2) \Phi(h)}
\sum_{\substack{\ord(\chi_1)=d_1 \\ \ord(\chi_2)=d_2 \\ \Ord(\psi)=h}}
\tilde{\bchi}_f (\chi_1,\chi_2,\psi).
\end{split}
\end{equation*}

From (ii) and (v) in the proof of Theorem \ref{principal} and from
$$
\displaystyle \sum_{\substack{d_1\mid \ell , d_2\mid \ell \\ P_j \mid h , h\mid 
P_j g}}
|\mu(d_1)| |\mu(d_2)| |\mu'(h)| = W(l)^2 W_{q}(g) 
$$
we conclude that
\begin{equation} \label{N3}
\left| N_f( \ell,\ell,P_j g) - \Phi(P_j) N_f(\ell, \ell,g) \right| \leq
(m_1 + m_2 +1) \dfrac{\Phi(g)}{N(g)} \dfrac{\Phi(P_j)}{N(P_j)} \theta(\ell)^2  
W(\ell)^2 W_{q}(g) q^{\frac{n}{2}}.
\end{equation}

So, replacing the results from (\ref{N1}), (\ref{N2}) and (\ref{N3}) in 
 (\ref{sieveineq}) we have that
\begin{equation*}
\begin{split}
N_f(& q^n-1, q^n-1, x^n-1)\\ &\geq  \delta N_{f}(\ell, \ell , g)  
 - ((m_1 + m_2) \theta(\ell)^2  \dfrac{\Phi(g)}{N(g)} W(\ell)^2 W_{q}(g) 
 q^{\frac{n}{2}}) (2\sum_{i=1}^{r} \theta(P_i) + \sum_{j=1}^{s} 
 \dfrac{\Phi(P_j)}{N(P_j)}) \\ 
 &= \delta N_{f}(\ell, \ell , g) \\
  & \textrm{\hspace{1.5cm}} -  ((m_1 + m_2) \theta(\ell)^2 
  \dfrac{\Phi(g)}{N(g)}  
 W(\ell)^2 
 W_{q}(g) 
 q^{\frac{n}{2}}) (2r - s - 2 \sum_{i=1}^{r} \dfrac{1}{p_i} - \sum_{j=1}^{s} 
 \dfrac{1}{q^{\deg P_j}})  \\
    &=\delta N_{f}(\ell, \ell , g) - ((m_1 + m_2) \theta(\ell)^2 
    \dfrac{\Phi(g)}{N(g)}  W(\ell)^2 W_{q}(g) q^{\frac{n}{2}})(\delta(\Delta - 
    1)).
\end{split}
\end{equation*}

From \eqref{eqpag10} we get
\begin{eqnarray*}
\displaystyle N_{f}(\ell,\ell,g) & > &\theta(\ell)^2 \dfrac{\Phi(g)}{ N(g)}
\left(
q^n-(m_1+m_2+1)- \right. \nonumber \\
& & \left. (m_1+m_2+1)q^{\frac{n}{2}}(W(\ell)^2W_q(g)-1)
\right).
\end{eqnarray*}

Thus,
\begin{eqnarray*}
N_f(q^n-1,q^n-1,x^n-1) & \geq  & \delta \theta(\ell)^2 \dfrac{\Phi(g)}{N(g)} [ q^n
- (m_1 + m_2 + 1) \Delta W(\ell)^2 W_{q}(g)  q^{\frac{n}{2}} \\
& &
 + (m_1 + m_2 + 1) q^{\frac{n}{2}} - (m_1 + m_2 + 1)  ].
\end{eqnarray*}

From the hypothesis we have $\delta>0$, and since $q^{\frac{n}{2}}(m_1 + m_2 + 
1) - 
(m_1 + m_2 + 1)>0$, we conclude that if $q^n > q^{}\frac{n}{2} W(\ell)^2 
W_{q}(g) (1 + m_1 + m_2)\Delta$, then $N_f(q^n-1,q^n-1,x^n-1)>0$.
\end{proof}

%\begin{eqnarray*}
%% \leq ((m_1 + m_2) \theta(\ell)^2  W(\ell)^2 W_{q}(g) q^{\frac{n}{2}}) 
%%(2\sum_{i=1}^{r} \theta(P_i) + \sum_{j=1}^{s} \Phi(P_j))\\
% = ((m_1 + m_2) \theta(\ell)^2  W(\ell)^2 W_{q}(g) q^{\frac{n}{2}}) (2r - s - 
%2 \sum_{i=1}^{r} \dfrac{1}{p_i} - \sum_{j=1}^{s} \dfrac{1}{q^{\deg P_j}})\\
%  = ((m_1 + m_2) \theta(\ell)^2  W(\ell)^2 W_{q}(g) 
%q^{\frac{n}{2}})(\delta(\Delta - 1)).
%\end{eqnarray*}

\begin{definition} \label{conjunto-B}
For  positive integers
$m_1$ and $m_2$ let  $\Bb(m_1,m_2)$
	be the set of pairs $(q,n) \in \mathbb{N}^2$, with $q$ a prime power, 
	such that for each $f \in \Upsilon_{q^n} (m_1,m_2)$ there exists a 
	primitive
	element
	$\alpha \in \F_{q^n}$, normal over $\F_q$, with
	$f(\alpha)$ primitive in $\F_{q^n}$.
\end{definition}
%Let 
%$$
%\Bb(m_1,m_2) = \bigcup_{p \text{ is prime}} \Bb_p(m_1,m_2).
%$$
Note that if $n_1\leq m_1$ and $n_2 \leq m_2$ then
$\Bb(m_1,m_2) \subset \Bb(n_1,n_2)$.
We finish this section by proving that 
there exists only a finite number
of pairs $(q,n) \in \N^2$ such that $q$ is a prime power and
$(q,n) \notin \Bb(m_1,m_2)$. 
For this,
we will need the following result, which is modeled after
\cite[Lemma 3.3]{CH} and
\cite[Lemma 4.1]{Kapetanakis-Reis} and, like these results, is proved using the 
multiplicativity of the function $W(\cdot)$ and the fact that if a positive 
integer $M$ has $s$ distinct prime divisors then $W(M) = 2^s$.

\begin{lemma}\label{cota-t}
	Let $M$ be a positive integer and $t$ be a positive real number.
	Then
	$W(M)
	\leq A_{t,M} \cdot M^{\frac{1}{t}}$,
	where
	$$
	A_{t,M}=\prod_{\substack{\wp < 2^t \\ \wp \text{ is prime}
	                  \\ \wp \mid M}}
	\frac{2}{\sqrt[t]{\wp}}.
	$$
\end{lemma}

\noindent
In many places we will use the possibly worse bound

\begin{equation}  \label{eq3} 
	W(M)
	\leq A_t \cdot M^{\frac{1}{t}},
\end{equation}
	where
	$$
	A_t = \prod_{\substack{\wp < 2^t \\ \wp \text{ is prime}}}
	\frac{2}{\sqrt[t]{\wp}}.
	$$

\noindent
We come to the  last result in this section.

\begin{proposition}\label{Asymptotic}
	There exists only a finite number
	of pairs $(q,n) \in \N^2$ such that $q$ is a prime power and
	$(q,n) \notin \Bb(m_1,m_2)$.
	%After
	%Let $q$ be a prime power and $n$ be a natural number.
	%If $n \geq 3$ and $q \geq 3.74\cdot 10^9$
	%or $n \geq 84$ and $q \geq 23$ then
	%$(q,n) \in \Bb(3,2)$.
\end{proposition}
\begin{proof}
Clearly every $\alpha \in \F_q^*$ is  
	normal over
	$\F_q$ and it is well known that  if $\alpha \in \F_{q^2}$ is 
	primitive, then
	$\alpha$ is also normal over $\F_q$. Thus, for $n=1$ or $n=2$ we get that
	$(q,n)\in \Bb(m_1,m_2)$ if and only if
	for every $f \in \Upsilon_{q^n} (m_1,m_2)$
	there exists a primitive element $\alpha \in \F_{q^n}$ such that 
	$f(\alpha)$ is
	also primitive. This problem was solved in \cite{CSS} and from \cite[Thm. 
	3.1]{CSS} we know that a sufficient condition for the existence of 
	such an element is that
	$q^{n/2} \geq (m_1 + m_2) W(q^n-1)^2$. Using Lemma \ref{cota-t}
	and choosing a real number $t>4$
	one may check that if $q \geq \left( 
	(m_1+m_2)A_t\right)^{\frac{2t}{(t-4)n}}$
	then $(q,n)\in \Bb(m_1,m_2)$. In particular there exists only a finite 
	number of pairs
	$(q,n) \notin \Bb(m_1,m_2)$ when $n=1$ or $n=2$.
	
	We assume now that $n \geq 3$, 
	clearly $W_q(x^n-1) \leq 2^n$ and using the bound \eqref{eq3} above 
	we have $W(q^n-1) \leq A_t \cdot q^{\frac{n}{t}}$ for any real number $t 
	>0$, so  
	from Corollary \ref{mainresult}, we get that if 
	$q^{\frac{n}{2}} \geq (m_1+m_2+1)\cdot A_t^2 \cdot q^{\frac{2n}{t}} \cdot 
	2^n$
	then $(q,n)\in \Bb(m_1,m_2)$.
%	, since
%	\begin{equation}\label{inequation}
%		(m_1+m_2+1)\cdot A_t^2 \cdot q^{\frac{2n}{t}} \cdot 2^n 
%		\geq (m_1 + m_2 + 1) W(q^n-1)^2W_q(x^n-1).
%	\end{equation}
	In particular, if we choose a real number $t>4$, 
	a condition to have $(q,n)\in \Bb(m_1,m_2)$
	for some $n \geq 3$ is
\begin{equation}\label{condition_q}
		q \geq 
		\left(
		2^n \cdot (m_1+m_2+1)\cdot A_t^2
		\right)^{\frac{2t}{(t-4)n}}.
\end{equation}
	In particular this means that for a given natural number $n$ there exists a 
	finite number of prime
	powers $q$ such that $(q,n)\notin \Bb(m_1,m_2)$.
	
	Inequality $q^{\frac{n}{2}} \geq (m_1+m_2+1)\cdot A_t^2 \cdot 
	q^{\frac{2n}{t}} \cdot 2^n$
	is also equivalent to
	\begin{equation}\label{condition-n}
		n \geq
		\frac{\ln \left( (m_1+m_2+1)\cdot A_t^2 \right)}%
		{(\frac{t-4}{2t}) \cdot  \ln q - \ln 2} ,
	\end{equation}
for $t> \frac{4 \ln q}{\ln q - 2 \ln 2}$.
The function on the right hand side is a decreasing function of
$q > 2^{\frac{2t}{t-4}}$.
If we choose $t\geq 29$ then the right hand side of \eqref{condition-n}
	is a decreasing function of $q \geq 5$.
	So, if
	$N$ is a natural number such that \eqref{condition-n} is true for $q=5$,
	for some $t\geq 29$, then
	$(q,n)\in \Bb(m_1,m_2)$ for all prime powers $q\geq 5$ and all
	natural numbers
	$n \geq N$.
	
	From \cite[Lemma 2.11]{Lenstra} we have for $n \geq 16$
	$$
	W_q(x^n-1) \leq 
	\left\{
	\begin{array}{ll}
	2^{\frac{n+5}{4}} & \text{if } q=2;\\
	2^{\frac{n+4}{3}} & \text{if } q=3;\\
	2^{\frac{n}{3}+2} & \text{if } q=4,
	\end{array}
	\right.
	$$
	and for these values of $q$ we may change inequality \eqref{condition-n} for
	\begin{equation}\label{condition-n-q4}
		n \geq
		\left\{
		\begin{array}{ll}
			\frac{4t}{t-8}
			\left(
			\frac{\ln ((m_1+m_2+1)\cdot A_t^2)}{\ln 2} + \frac{5}{4}
			\right)
				& \text{if } q=2, \text{ for some } t > 8; \\
			\frac{\ln ((m_1+m_2+1)\cdot A_t^2) + \frac{4}{3} \ln 2}%
			{\left( \frac{t-4}{2t} \right) \ln 3 - \frac{1}{3} \ln 2} 
			& \text{if } q=3, \text{ for some } t \geq 7;\\
			\frac{3t}{t-6}
			\left(
			\frac{\ln ( (m_1+m_2+1)\cdot A_t^2)}{\ln 4} +1
			\right)
			& \text{if } q=4, \text{ for some } t > 6.
		\end{array}
		\right.
	\end{equation}
	Putting \eqref{condition-n} and \eqref{condition-n-q4} together we get that 
	there exists a positive integer $M$ such that if $n \geq M$ then
	$(q,n)\in \Bb(m_1,m_2)$ for every prime power $q$. For the 
	natural numbers
	$n<M$ we get, from what was done above, that there exists a finite number 
	of prime 
	powers
	such that $(q,n)\notin \Bb(m_1,m_2)$.
\end{proof}

\section{Numerical examples}

%Let $(m_1,m_2)=(3,2)$ and 
%$p=2$, $3$ or $5$, 
%in this section we determine pairs $(q,n)$ with $n \geq 3$ such that  $(q,n) 
%\in \Bb_p(m_1,m_2)$.

%Let $(m_1,m_2)$ be a couple of natural numbers. We will determine,
%for $n \geq 3$, the prime powers $q$ such that$(q,n)  \in \Bb(m_1,m_2)$ 
%for some values of  $(m_1,m_2)$.
%We begin with the case $(m_1,m_2)=(3,2)$.

In this section we will determine $\Bb(3,2)$, except possibly for a finite 
number of pairs $(q,n)$, which we list.

\begin{proposition}\label{casen}
Let $q$ be a prime power and $n\geq 3$ be a natural number. 
We have $(q,n)  \in \Bb(3,2)$ for 
$q\geq 3.74 \cdot 10^9$ and $n=3$,
for $q\geq 3.91 \cdot 10^7$ and $n=4$,
for $q\geq2.5 \cdot 10^6$ and $n=5$, and
for $q\geq 23$ and $n\geq 6$.
\end{proposition}
\begin{proof}
Recall, from 
inequality \eqref{condition_q}, that 
if $q \geq \left(
2 \cdot \sqrt[n]{6 \cdot A_t^2}
\right)^{\frac{2t}{t-4}}$
holds for some real number $t > 4$ then 
 $(q,n) \in \mathcal{B}(3,2)$. Using this inequality we 
constructed the table below, where each row presents
values of $t_0$, $n$ and $M(n,t_0)$ such that 
taking $t = t_0$ we get $(q,n) \in \mathcal{B}(3,2)$ whenever
$q \geq M(n,t_0)$.

\vspace{2ex}
\begin{center}
\begin{tabular}{|l|l|l|}
\hline 
$t_0$ & $n$ & $M(n,t_0)$ \\ 
\hline 
6.3 & 3 & $3.74 \cdot 10^9$ \\ 
\hline 
6.3 & 4 & $3.91 \cdot 10^7$ \\ 
\hline 
6.4 & 5 &  $2.5 \cdot 10^6$  \\ 
\hline 
6.5 & 6, 7, 8, 9 & 394155 \\ 
\hline 
6.7 & $10, 11, \cdots, 157$ & 9239 \\ 
\hline 
9 & $n \geq 158$ & 23 \\ 
\hline 
\end{tabular}
\end{center} 

\vspace{2ex}
%If we take $t=6.3$ for $n=3$ and $n=4$ we get
%$(q,3) \in \Bb(3,2)$ for $q \geq 3.74 \cdot 10^9$,
%and
%$(q,4) \in \Bb(3,2)$ for $q \geq 3.91 \cdot 10^7$,
%if we take $t=6.4$ for $n=5$ we get
%$(q,5) \in \Bb(3,2)$ for $q \geq 2.6 \cdot 10^6$,
%if we take $t=6.5$ for $n\geq 6$ we get
%$(q,n) \in \Bb(3,2)$ for $q \geq 394155$ and $n \geq 6$,
%if we take $t=6.7$ for $n\geq 10$ we get
%$(q,n) \in \Bb(3,2)$ for $q \geq 9239$ and $n \geq 10$.
%Finally
%if we take $t=9$ for $n\geq 158$ we get
%$(q,n) \in \Bb(3,2)$ for $q \geq 23$ and $n \geq 158$.
Now we use SageMath to refine the lower bound for $q$ when $n \geq 6$.
We start by testing inequality 
$q^{\frac{n}{2}} \geq  6 W(\ell)^2W_q(g) \Delta$, which appears in Lemma 
\ref{divisores}, taking prime powers $q$ in the range 
$9239\leq q < 394155$, $\ell = \gcd(q^n-1,2 \cdot 3 \cdot 5 \cdot 7)$, $g = 1$ 
and $n \in \{6,7,8,9\}$ and we get that $(q,n)\in \Bb(3,2)$ for all these 
values of 
$q$ and $n$. Next, we combine the condition in Corollary \ref{mainresult} with 
the bound in Lemma \ref{cota-t} to obtain the inequality  
$q^{\frac{n}{2}} \geq 6 A_{t}^2 q^{\frac{2n}{t}} W_q(x^n-1)$
which holds true for prime powers
$23\leq q < 9239$ and $65 \leq n \leq 157$ when we take $t = 7$, so we get 
$(q,n)\in \Bb(3,2)$ for these values of 
$q$ and $n$. Taking $t = 8$ the inequality holds true for prime powers $23\leq 
q < 9239$ and $n$ in the range $6 \leq n < 65$, except for $7713$ pairs $(q,n)$.
To test if these pairs belong to $\Bb(3,2)$  we use again Lemma 
\ref{divisores}, taking $\ell=\gcd(q^n-1,2\cdot 3 \cdot 5 \cdot 7)$ and $g=1$, 
and we get an affirmative answer except for the pairs 
\[
\begin{array}{l}
(32, 31),(27, 26),(27, 52),(25, 24),(25, 48),(49, 48),(23, 22),\\
(23, 44),(31, 30),(37, 36),(41, 40),(43, 42),(47, 46),(53, 52).
\end{array}
\]
To check if these pairs belong to  $\Bb(3,2)$ we use one more time the 
condition in  
Lemma \ref{divisores}, again with $\ell=\gcd(q^n-1,2\cdot 3 \cdot 5 \cdot 7)$ 
but this time we take $g$ to be the product of all linear factors of $x^n - 1$, 
and we get that all these pairs belong to $\Bb(3,2)$. 
\end{proof}

The following result will be useful in the study of the case $n = 3$;

\begin{lemma}\label{div3}
Let $q$ be a prime power.
If $\wp \neq 3$ is a prime number
such that $\wp \mid q^2+q+1$, then
$\wp \nmid q-1$ and
$\wp \equiv 1 \pmod 3$.
\end{lemma}
\begin{proof}
From $\gcd(q-1,q^2+q+1)=\gcd(q-1,3) \in \{1,3\}$ and the fact that 
 $\wp \neq 3$ is a prime divisor of $q^2+q+1$ we get 
$\wp \nmid q-1$. 
From 
$\wp \mid q^3-1$ and $q \not\equiv 1 \pmod \wp$ we get that there is an element 
of order 3 in 
$\mathbb{F}_{\wp}^*$ so $3 \mid \wp - 1$.
\end{proof}

\begin{proposition}\label{casen3}
Let $q$ be a prime power. Then
$(q,3)\in \Bb(3,2)$, except possibly for
\begin{eqnarray}
q \in \{ 2,4,8,16,3,9,27,81,5,25,7,49,11,121,13,17, \nonumber\\
19,23,29,31,37,43,61,67,71,79,151,211,331\}. \nonumber
\end{eqnarray}
\end{proposition}
\begin{proof}
From Lemma \ref{casen} we have
$(q,3)  \in \Bb(3,2)$ for 
$q\geq 3.74 \cdot 10^9$, so we assume 
$q< 3.74 \cdot 10^9$. Let
$$
q^2+q+1 = 3^{a_0} \cdot \prod_{i=1}^r \wp_i^{a_i}
$$
be the factorization in distinct prime numbers of $q^2+q+1$. From
Lemma \ref{div3} we have $\wp_i \nmid q-1$ and
$\wp_i \equiv 1 \pmod 3$ for $i \in \{1,\ldots , r \}$.
For any natural number $k$ let  $S_k$ and $P_k$ be, respectively, the sum of 
the inverses
and the product of the first $k$ primes of the form $3j+1$.
Then 
$P_r \leq q^2+q+1 < 1.3988\cdot 10^{19}$
and since $P_{11} < 3 \cdot 10^{17}$ and $3.26 \cdot 10^{19} < P_{12}$
we must have $r \leq 11$.
To check if $(q,3) \in \Bb(3,2)$ we use the condition 
$q^{\frac{n}{2}} \geq 6 W(\ell)^2 W_q(g) \Delta$
which appears in  Lemma \ref{divisores},
taking $\ell=q-1$ and $g=1$. Let's suppose that $q > 10^4$, then 
$$
\delta \geq 1 - 2 S_r - \frac{3}{q} >1 - 2 S_{11} - \frac{3}{10^4} 
> 0.153
$$
and $\Delta=2+\frac{2r+s-1}{\delta} <  2+\frac{2 \cdot 11+3-1}{0.153}< 159$.
Using the bound \eqref{eq3} we get that the above condition is satisfied if
$$
q^{\frac{3}{2}} \geq 6 \cdot A_t^2 \cdot q^{\frac{2}{t}} \cdot  W_q(1) \cdot 
159,
$$
or equivalently, if  $q \geq \left(954 \cdot A_t^2 
\right)^{\frac{2t}{3t-4}}$ for some real number $t > 0$.
Taking $t=3.7$ we get $(q,3) \in \Bb(3,2)$
for $q \geq 22282$.
To verify if $(q,3) \in \Bb(3,2)$ for the  
 $2563$ prime powers less than $22282$ we used SageMath to test  
the condition in 
Lemma \ref{divisores}, taking
$\ell=\gcd(q^3-1,2\cdot 3 \cdot 5)$ and $g=1$, and we got that
$(q,3)\in \mathcal{B}(3,2)$ for all prime power $q$ above except possibly for 
\begin{eqnarray}
q \in \{ 2,4,8,16,3,9,27,81,5,25,7,49,11,121,13,17,19, \nonumber \\
23,29,31,37,41,43,61,67,71,79,151,181,211,331\}. \nonumber
\end{eqnarray}

We also got $(41, 3) \in \Bb(3,2)$ (respectively, $(181, 3) \in \Bb(3,2)$) by 
using SageMath to test the condition in Lemma \ref{divisores} with the data 
$\ell=2$, $g=1$ (respectively, $\ell=2 \cdot 3$ and $g=1$).
%
%We also may use Lemma \ref{divisores} for $q=41$ and $q=181$.
%For $q=41$ and $n=3$ we take
%$\ell=2$, $g=1$,
%where $q^n-1=2^3 \cdot 5 \cdot 1723$ and
%$x^n-1=(x - 1) \cdot (x^2 + x + 1)$.
%For $q=181$ and $n=3$ we take $\ell=2 \cdot 3$ and $g=1$,
%and use that $q^n-1=2^2 \cdot 3^3 \cdot 5 \cdot 79 \cdot 139$ and
%$x^n-1=(x - 1) (x - 48) (x-132)$.
\end{proof}

\begin{proposition}\label{casen4}
Let $q$ be a prime power. Then
$(q,4)\in \Bb(3,2)$, except possibly for
\begin{eqnarray}
q \in \{ 2,4,8,16,3,9,27,5,25,7,11,13, \nonumber\\
17,19,23,29,31,37,41,43,47,83\}. \nonumber
\end{eqnarray}
\end{proposition}
\begin{proof}
From Lemma \ref{casen} we have
	$(q,4)  \in \Bb(3,2)$ for 
	$q\geq 3.91 \cdot 10^7$
so we assume
$q< 3.91 \cdot 10^7$.
Let
$$
q^4-1 = 2^{a_0} \cdot 3^{a_1} \cdot 5^{a_2} \cdot 7^{a_3} \cdot 11^{a_4} \cdot 13^{a_5}
\cdot
 \prod_{i=1}^r \wp_i^{b_i}
$$
be the factorization in distinct prime numbers of $q^4-1$,
where $\wp_i > 13$ for $i \in \{1,\ldots , r \}$.
Let define $S_k$ and $P_k$ be, respectively, the sum of the inverses
and the product of the first $k$ odd primes numbers bigger than $13$.
We have 
$P_r \leq q^4-1 < 2.34\cdot 10^{30}$ and from $P_{18} < 7.92 \cdot 10^{29}$ and 
$7.67 \cdot 10^{31} < P_{19}$ we get that $r \leq 18$.
Let's suppose that $q > 10^3$. We want to apply Lemma \ref{divisores} with 
$
\ell=\gcd (q^4-1, 2\cdot 3 \cdot 5 \cdot 7 \cdot 11 \cdot 13)
$
and $g = 1$. We have 
$$
\delta \geq 1 - 2 S_r - \frac{4}{q} > 1 - 2 S_{18} - \frac{4}{10^3} 
> 0.099
$$
so 
$\Delta=2+\frac{2r+s-1}{\delta} <  2+\frac{2\cdot 18+4-1}{0.099} < 396.$
Since 
$W(\ell)\leq 2^6$ and $W_q(1)=1$ we get that if $q \geq 3120 \geq (6 \cdot 
(2^6)^2 \cdot  
396)^{\frac{1}{2}}$ then 
$q^{\frac{4}{2}} \geq 6 W(\ell)^2 W_q(g) \Delta$ and from Lemma \ref{divisores} 
we get that $(q,4) \in \Bb(3,2)$.
There are $480$ prime powers lesser than $3120$, and using SageMath and 
Lemma \ref{divisores}, with 
$\ell=\gcd(q^4-1,2\cdot 3 \cdot 5)$ and $g=1$, to test them we get
$(q,4)\in \mathcal{B}(3,2)$ for all such prime powers except for 
\begin{eqnarray}
q \in \{ 2,4,8,16,32,3,9,27,5,25,7,49,11,13,17, \nonumber\\
19,23,29,31,37,41,43,47,53,61,67,73,83\}. \nonumber
\end{eqnarray}	
For these values, we use again SageMath and
Lemma \ref{divisores}, taking
$\ell=\gcd(q^4-1,2\cdot 3)$ and $g=1$, and we get
$(q,4)\in \Bb(3,2)$ for
$q \in \{32, 49, 53, 61, 67, 73\}$.  
\end{proof}

\begin{proposition}\label{casen5}
Let $q$ be a prime power. Then
$(q,5)\in \Bb(3,2)$, except possibly for
$q \in \{ 2, 3, 4, 5, 7, 9, 11, 16\}$.
\end{proposition}
\begin{proof}
From Lemma \ref{casen} we have
$(q,5)  \in \Bb(3,2)$ for 
$q\geq 2.5 \cdot 10^6$
so we assume that 
$q < 2.5 \cdot 10^6$.
There are $183404$ prime powers less than $2.5 \cdot 10^6$, and 
for those prime powers
we use  
Lemma \ref{divisores}, taking
$\ell=\gcd(q^5-1,2\cdot 3 \cdot 5)$ and $g=1$, together with SageMath  and we 
get that 
$(q,5)\in \mathcal{B}(3,2)$ for all prime powers $q$ except for 
$q \in \{ 2, 3, 4, 5, 7, 9, 11, 16, 31\}$.
We also get $(31,5) \in \Bb(3,2)$, using again Lemma \ref{divisores},
where we take
$\ell=2\cdot 3$ and $g=1$, together with SageMath.
\end{proof}

Now we will treat the cases where $2 \leq q \leq 19$.
In what follows we will frequently want to bound $W(M)$ with $M = q^n - 1$, 
where $q$ is a power of a prime $p$. Since $p \nmid q^n - 1$, instead of  
the number $A_t$, which appears in inequality \eqref{eq3} we will use
$$
	\tA_{t,p}:=A_{t,q^n-1}=\prod_{\substack{\wp < 2^t \\ \wp \neq p \\
			\wp \text{ is prime}}}
	\frac{2}{\sqrt[t]{\wp}} 
$$
and clearly we have 
\begin{equation*}
	W(q^n -1)
	\leq \tA_{t,p} \cdot (q^n - 1)^{\frac{1}{t}},
\end{equation*}

\begin{proposition}\label{case2_1}
For $q=2$ and $n \geq 3$ we have $(2,n) \in \mathcal{B}(3,2)$
for every $n \in \{13,17,19,21,22,23 \}$ and all $n \geq 25$.
\end{proposition}
\begin{proof}
%	From \cite[Proposition 4.1]{CGJV} we have that
%	$\N \backslash \{1,2,3,4,6,8,10,12\} \subset \Gamma_2(3,2)$
%	and $\{ 1,2,3,4 \} \cap  \Gamma_2(3,2) =\emptyset$.
%	This means that $(q,1) \in \Bb_2(3,2)$ 
%	where $q$ is a power of $2$ such that
%	$q\geq 2^{13}$ or $q \in \{2^5,2^7,2^9,2^{11}\}$ and 
%	$(q,1) \not\in \Bb_2(3,2)$ for $q=2$, $q=2^2$, $q=2^3$ and $q=2^4$.
%	We also have 
%	$(q,2) \in \Bb_2(3,2)$ where $q$ is a power of $2$ such that
%	$q\geq 2^{7}$  and 
%	$(q,2) \not\in \Bb_2(3,2)$ for $q=2$ and $q=2^2$.
To test if $(2,n) \in \mathcal{B}(3,2)$ we start by proceeding as in the second 
paragraph of the proof of Proposition \ref{Asymptotic}, with the difference 
that, instead of using $A_t$  we use the number $\tA_{t,2}$ introduced above
$$
	\tA_{t,2}=\prod_{\substack{\wp < 2^t \\ \wp \neq 2 \\
			\wp \text{ is prime}}}
	\frac{2}{\sqrt[t]{\wp}} . 
	$$
Thus we have $W(2^n-1) \leq \tA_{t,2} \cdot 
2^{\frac{n}{t}}$ and following the proof
we get that if for 
some $t > 8$ we have
\[
n > \frac{4t}{t-8}
			\left(
			\frac{\ln (6 \cdot \tA_{t,2}^2)}{\ln 2} + \frac{5}{4}
			\right)
\]
(see \eqref{condition-n-q4}) then $(2,n) \in \mathcal{B}(3,2)$, and taking 
$t = 9.8$ we get that $(2,n) \in \mathcal{B}(3,2)$ for $n \geq 1237$. Next, 
using SageMath, we get that the inequality 
$2^{\frac{n}{2}} \geq 6 \cdot \tA_{t,2}^2 \cdot 2^{\frac{2n}{t}} \cdot 
W_2(x^n-1)$,
when we take $t=8.1$, holds for 
$156 \leq n \leq 1236$, so from Corollary \ref{mainresult} we also have $(2,n) 
\in \mathcal{B}(3,2)$ for these values of $n$.  Now we use SageMath to test the 
inequality  
$2^{\frac{n}{2}} \geq 6 \cdot W(2^n-1)^2 \cdot W_2(x^n-1)$ of Corollary 
\ref{mainresult} and we get that it 
holds for $n \geq 17$, except  $n \in \{ 18,20,21,22,24,28,30,36,45\}$. 
For these exceptions and $n$ such that $3 \leq n \leq 16$ we use SageMath and 
Lemma \ref{divisores} with $\ell=\gcd (2^n-1, 3 \cdot 5)$ and $g$ as the product 
of the monic irreducible factors of $x^n-1$ of degree $k$, with
$2^k\leq 2\cdot n$  (in order to increase the chance of getting a positive 
$\delta$) and we get $(2,n) \in \mathcal{B}(3,2)$
for $n \in \{13, 21, 22, 28, 30, 36, 45 \}$.
\end{proof}

\begin{proposition}\label{case3_1}
For $q=3$ and $n \geq 3$, we have $(3,n) \in \mathcal{B}(3,2)$
for every $n \in \{11,13,14,15\}$ and all $n \geq 17$.
\end{proposition}
\begin{proof}
As we did above, we proceed as in the second paragraph of 
the proof of Proposition \ref{Asymptotic}, and since $3 \nmid 3^n - 1$,
instead of $A_t$ we use $\tA_{t,3}$ so that  
$W(3^n-1) \leq \tA_{t,3} \cdot 3^\frac{n}{t}$ holds, with $t > 0$.
From inequality \eqref{condition-n-q4} for the case $q = 3$, and taking $t = 
8.8$, we get $(3,n)\in \Bb(3,2)$ for $n \geq 373$. 
For $122 \leq n \leq 372$ calculations with SageMath show that the inequality 
$3^{\frac{n}{2}} \geq 6 \cdot \tA_{t,3}^2 \cdot 3^{\frac{2n}{t}} \cdot 
W_3(x^n-1)$ 
holds if we take $t = 8$, so from Corollary \ref{mainresult} we get $(3,n)
\in \mathcal{B}(3,2)$ for $n$ in this range.
For $3 \leq n \leq 121$ and we use SageMath  to test inequality
$3^{\frac{n}{2}} \geq 6 \cdot W(3^n - 1)^2 \cdot W_3(x^n-1)$
of Corollary \ref{mainresult} and we get that $(3,n)\in \Bb(3,2)$ 
for $17 \leq n \leq 121$ except for $n \in \{ 18,20,22,24\}$. For $3 \leq n 
\leq 16$ and $n \in \{ 18,20,22,24\}$ we use SageMath to test the inequality 
which appears in Lemma \ref{divisores}, taking 
$\ell=\gcd (3^n-1, 2 \cdot 5)$ and $g$ as the product of 
the monic irreducible linear factors of $x^n-1$ and we get  $(3,n)\in 
\Bb(3,2)$ for $n \in \{11, 13, 14, 15, 18, 20, 22, 24 \}$.
\end{proof}

\begin{proposition}\label{case2_2}
For $q=4$ and $n \geq 3$, we have $(4,n) \in \mathcal{B}(3,2)$
for $n=11$ and all $n \geq 13$.
\end{proposition}
\begin{proof}
Again we proceed as in the proof of Proposition \ref{Asymptotic}, again 
taking $\tA_{t,2}$ in place of $A_t$, and from inequality
\eqref{condition-n-q4} for the case $q = 4$,  taking 
$t=8$, we get $(4,n)\in \Bb(3,2)$
for $n \geq 163$. Using SageMath to test inequality
$4^{\frac{n}{2}} \geq 6 \cdot \tA_{t,2}^2 \cdot 4^{\frac{2n}{t}} \cdot W_4(x^n-1)$
with $t=7$
and we get that $(4, n) \in  \mathcal{B}(3,2)$ for $86 \leq n \leq 162$.
Now we use SageMath we test the inequality 
$4^{\frac{n}{2}} \geq 6 \cdot W(4^n-1)^2 \cdot W_4(x^n-1)$, from Corollary 
\ref{mainresult} 
for $3 \leq n \leq 85$ and we get
that $(4,n)\in \Bb(3,2)$ for
$n\geq 13$, except possibly for  $n \in \{ 14, 15, 18, 21, 30\}$.
Finally we use SageMath to test the condition
$4^{\frac{n}{2}} \geq 6 W(\ell)^2W_4(g) \Delta$
from Lemma \ref{divisores}, taking 
$\ell=\gcd(4^n-1, 3\cdot 5\cdot 7)$ and $g$ as the product of 
the linear factors of $x^n-1$, and we get 
$(4,n)\in \Bb(3,2)$  for $n \in \{11, 14, 15, 18, 21, 30\}$.
\end{proof}

\begin{proposition}\label{case5_1}
	For $q=5$ and $n \geq 3$, we have $(5,n) \in \mathcal{B}(3,2)$
	for all $n \geq 13$ and for
	$n \in \{7,9,10,11\}$.
\end{proposition}
\begin{proof}
From 	
	\cite[Lemma 2.11]{Lenstra} we have 
	$W_5(x^n-1) \leq 
	2^{\frac{n}{3}+6}$ and from Lemma \ref{cota-t} we have
	$W(5^n-1)\leq \tA_{t,5} \cdot 5^{\frac{n}{t}}$, with  $t > 0$ a real number.
	From Theorem \ref{principal} we get that  if
	$5^{\frac{n}{2}} \geq 6 W(5^{n}-1)^2 W_5(x^n-1)$
	then $(5,n) \in \mathcal{B}(3,2)$. Hence 
	$(5,n) \in \mathcal{B}(3,2)$ if $n$ satisfies
$$
	5^{\frac{n}{2}} \geq 6 \cdot \tA_{t,5}^2 \cdot 5^{\frac{2n}{t}} \cdot 2^{\frac{n}{3}+6}
$$
for some real number $t > 0$. 
Last inequality is equivalent to
\begin{eqnarray}\label{cond5_1}
n \geq
\frac{\ln \left( 6 \tA_{t,5}^2 \right)+ 6 \ln 2}%
{(\frac{t-4}{2t})  \ln 5 - \frac{1}{3} \ln 2} .
\end{eqnarray}
when $(\frac{t-4}{2t})  \ln 5 - \frac{1}{3} \ln 2 > 0$, which holds e.g.\ if $t 
> 5.62$.
For $t=7.8$ we get that 
\eqref{cond5_1} is satisfied for
$n\geq 127$.
Next we use SageMath to test the condition
$5^{\frac{n}{2}} \geq 6 W(5^{n}-1)^2 W_5(x^n-1)$ of Corollary \ref{mainresult} 
for $3 \leq n\leq 126$ and we get that $(5,n) \in \Bb(3,2)$
for all
$n \geq 25$ and for
	$n \in \{ 11,13,15,17,19,20,21,22,23\}$.  
Then we	use SageMath to test the condition
$q^{\frac{n}{2}} \geq 6 \cdot W(\ell)^2 \cdot W_5(g) \cdot \Delta$
of Lemma \ref{divisores}, with 
$\ell=\gcd(q^n-1, 2\cdot 3)$ and $g$ as the product of 
the linear factors of $x^n-1$, and we get 
$(5,n) \in \mathcal{B}(3,2)$ for 
$n \in \{7,9,14,16,18,24\}$.
The same inequality with $n = 10$, $\ell=2\cdot 3$ and  $g=1$,
shows that $(5,10) \in \mathcal{B}(3,2)$.
\end{proof}

\begin{proposition}\label{caseq_low}
Let $q\in \{ 7, 8, 9, 11, 13, 16, 17, 19\}$
and $n\geq 3$.
We have $(q,n) \in \mathcal{B}(3,2)$ except possibly for
$$
\begin{array}{l}
(7, 3),(7, 4),(7, 5),(7, 6),(7, 8),(7, 12),(8, 3),(8, 4),(8, 7),
(9, 3),(9, 4),\\(9, 5),(9, 6),(9, 8),(11, 3),(11, 4),(11, 5),(11, 6),
(13, 3),(13, 4),(13, 6),\\(16, 3),(16, 4),(16, 5),
(17, 3),(17, 4),(19, 3),(19, 4),(19, 6).
\end{array}
$$
\end{proposition}
\begin{proof}
Let $s$ be the number of monic irreducible factors of 
	$x^n-1 \in \F_q[x]$, then $W_q(x^n-1) = 2^s$ and  
	from \cite[Inequality (2.10)]{Lenstra} we have 
	$$
	s \leq \frac{1}{2} \left( n + \gcd(n,q-1) \right).
	$$
Let's assume that $n \geq 19$, then $\gcd(n,q-1) \leq n/2$ and we get  
$W_q(x^n-1) \leq 2^{\frac{3n}{4}}$. Let 
$p=\mathop{\rm{char}}(\F_q)$, 
 from  Lemma \ref{cota-t} and the fact  that $p \nmid q^n - 1$ we get 
$W(q^n - 1)\leq \tA_{t,p} \cdot (q^n - 1)^{\frac{1}{t}}$ and from 
Corollary \ref{mainresult} we have that 
if 
$$
q^{\frac{n}{2}} \geq 6 \cdot \tA_{t,p}^2 \cdot q^{\frac{2n}{t}} \cdot 
2^{\frac{3n}{4}}
$$
for some real number $t > 0$ then $(q,n) \in \mathcal{B}(3,2)$.
The above inequality is equivalent to
\begin{eqnarray}\label{condq_2}
n \geq
\frac{\ln \left( 6 \cdot \tA_{t.p}^2 \right)}%
{(\frac{t-4}{2t})  \ln q - \frac{3}{4} \ln 2} .
\end{eqnarray}
provided that $(\frac{t-4}{2t})  \ln q - \frac{3}{4} \ln 2 > 0$.
In Table \ref{Caseq_1} we show values of $q$ and ranges of $n$ for which this 
inequality 
holds, together with the values of $t$ which were used in each case.
\begin{table}[h]
\centering
\begin{tabular}{ccc|ccc|ccc|ccc}
$t$ & $q$ & $n$ & $t$  &$q$ & $n$ & $t$  &$q$ & $n$ & $t$  &$q$ & $n$\\
\hline
$10.4$ & $7$  & $n \geq 649$ &
        $9.4$& $9$ & $n \geq 289$  
			& $8.6$ &  $13$  & $n \geq 138$ 
				& $8.1$ &  $17$  & $n \geq 95$ \\
$9.8$ & $8$  & $n \geq 403$ & 
		$9$ & $11$ & $n \geq 186$ &
			$8.1$ & $16$ & $n \geq 100$ &
				$8$ & $19$ & $n \geq 84$
\end{tabular}
\caption{Values of $q$, $n$ and a real number $t$ for which
$(q,n)\in \mathcal{B}(3,2)$}
\label{Caseq_1}
\end{table}

\noindent
For $7 \leq q \leq 19$ we test condition
$q^{\frac{n}{2}} \geq 6 \tA_{t,p}^2 q^{\frac{2n}{t}} W_q(x^n-1)$,
with $\tA_{t,p}$ as in \eqref{eq3}
and $t=8$, for pairs $(q,n)$ which are not in Table \ref{Caseq_1}
and for those pairs where this condition does not hold we test the inequality
$q^{\frac{n}{2}} \geq (m_1 + m_2+1)  W(\ell)^2 W_q(g)  \Delta$
of Lemma \ref{divisores},
taking $\ell=\gcd(q^n-1, 2\cdot 3 \cdot 5)$ and $g$ as the product of 
the linear factors of $x^n-1$. In this way we get $(q,n)\in \Bb(3,2)$, with 
$n\geq 3$, except for 
the following pairs:
$$
\begin{array}{l}
(7, 3),(7, 4),(7, 5),(7, 6),(7, 7),(7, 8),(7, 9),(7, 12),\\
(8, 3),(8, 4),(8, 6),(8, 7),(9, 3),(9, 4),(9, 5),(9, 6),(9, 8),\\
(11, 3),(11, 4),(11, 5),(11, 6),(11, 10),\\
(13, 3),(13, 4),(13, 5),(13, 6),(13, 8),(13, 12),\\
(16, 3),(16, 4),(16, 5),(16, 6),(16, 45),\\
(17, 3),(17, 4),(17, 6),(17, 8),(19, 3),(19, 4),(19, 6).
\end{array}
$$
Now, for this set of pairs, we use again Lemma \ref{divisores}
with $\ell=\gcd(q^n-1, 2\cdot 3 \cdot 5)$ and $g=1$
and we get $(q,n) \in \Bb(3,2)$ for 
$$
(q,n) \in \{ (7, 7), (7, 9), (8, 6), (11, 10), (13, 5), (13, 8), (16, 6), (17, 6), (17, 8)\}.
$$
We also get  $(13,12),(16,45) \in \Bb(3,2)$ using Lemma \ref{divisores}:
for $q=13$ and $n=12$ we take
$\ell=2\cdot 3\cdot 5 \cdot 7$ and $g=(x-1)(x+1)$,
and for $q=16$ and $n=45$ we take
$\ell=3 \cdot 5 \cdot 7$ and $g$ as the product
of all linear factors of $x^n-1$.
\end{proof}

The following results
are modeled after \cite[Propositions 3.3 and 3.4]{CGJV}, and will be needed
in what follows.
For a prime power $q$ and $n \in \mathbb{N}$ we denote by 
$\mathfrak{N}(q,n)$ the number of primitive elements 
of $\mathbb{F}_{q^n}$ which are normal over 
$\mathbb{F}_q$.

%$$PN(q,n):=\{ \alpha \in \mathbb{F}_{q^n} \mbox{ ; } \alpha \mbox{ is 
%primitive 
%and normal over } \mathbb{F}_q\}
%$$
%and $\mathfrak{N}(q,n):=\left| PN(q,n) \right|$.

\begin{proposition}\label{noBb}
Let $q$ be a prime power
and $n\geq 3$. 
If $\mathfrak{N}(q,n) \leq m_1 + m_2+1$, then $(q,n) \notin \Bb(m_1,m_2)$.
\end{proposition}
\begin{proof}
Let $ \{ \alpha_1, \ldots , \alpha_N \}$ be the 
set of all primitive elements 
of $\mathbb{F}_{q^n}$ which are normal over $\mathbb{F}_{q^n}$. Note that, if 
$N \leq 
m_1 + m_2+1$, we
may choose polynomials $f_1(x)$ and $f_2(x)$
of degrees $m_1$ and $m_2$, respectively, such that $f_1(\alpha_j) 
f_2(\alpha_j)=0$, for all $j=1,\ldots,N - 1$,  $f_1(\alpha_N)  f_2(\alpha_N 
) \neq 0$, and $f(x) = \frac{f_1(x)}{f_2(x)} \in \Upsilon_{q} (m_1,m_2)$. 
Thus,  
$f(\alpha_j)$ is not primitive for all $j=1,\ldots,N-1$ and taking $\beta = 
\frac{1}{ f( \alpha_N ) }$ we have that 
 $h(x)=\beta f(x) \in \Upsilon_{q} (m_1,m_2)$ and $h(\alpha_N) = 1$ is not 
 primitive either. This proves that $(q,n) \notin \Bb(m_1,m_2)$.
\end{proof}

\begin{proposition}\label{yesBb}
Let $q=2^k$, $n\geq 3$ and $m= max \{ m_1 , m_2\}$.
If
$$\dfrac{\mathfrak{N}(q,n) }{m} + \phi(q^n-1)> q^n +1,
$$
then $(q,n)\in \Bb(m_1,m_2)$.
\end{proposition}
\begin{proof}
Let $f(x) = f_1(x)/f_2(x) \in \Upsilon_{q^n} (m_1,m_2)$ and  let 
\[
A_f = \{ \alpha \in \mathbb{F}_{q^n} ;  \alpha \textrm{ primitive and normal 
over }
\mathbb{F}_q \textrm{ and } f_2(\alpha) \neq 0\}.
\]
Clearly $|A_f| \geq \mathfrak{N}(q,n) - m_2 \geq \mathfrak{N}(q,n) - m$.  
Let $\widetilde{f} : A_f \rightarrow \mathbb{F}_{q^n}$ be defined by $\alpha 
 \mapsto 
 f(\alpha)$,  given 
 $\beta \in \widetilde{f}(A_f)$ there are at
	most $m$
	elements $\alpha \in A_f$
	such that $f(\alpha)=\beta$, since $\alpha$ must be a zero of the
	polynomial $f_1(x)-\beta f_2(x)$,
	hence 
\[
|\widetilde{f}(A_f)| \geq \dfrac{\mathfrak{N}(q,n) - m}{m} = 
\dfrac{\mathfrak{N}(q,n)}{m} - 1.
\] 
There are $\phi(q^n -1)$
primitive elements in $\mathbb{F}_{q^n}$, therefore if $ 
\dfrac{\mathfrak{N}(q,n)}{m} - 1 + \phi(q^n-1)> q^n$, then at least one element 
$f(\alpha)$ is primitive for some $\alpha \in A_f$ and the result 
follows.
\end{proof}

\begin{lemma}
We have $(2,3),(2,4),(2,6),(3,3),(3,4) \notin \Bb(3,2)$ and 
$(2,5)$, $(2,7)$, $(2,11)$, $(8,3) \in \Bb(3,2)$.
\end{lemma}
\begin{proof}
Using SageMath we determined that $\mathfrak{N}(2,3)=3$ and 
$\mathfrak{N}(2,4)=4$, so from Proposition \ref{noBb}
we get $(2,3),(2,4) \notin \Bb(3,2)$. We also determined that
$\mathfrak{N}(2,5)=15$, $\mathfrak{N}(2,7)=49$, $\mathfrak{N}(2,11)=957$ and
$\mathfrak{N}(8,3)=378$ so from Proposition \ref{yesBb} we get
$(2,5),(2,7),(2,11),(8,3) \in \Bb(3,2)$.

Using also SageMath we found that
for $(q,n)=(2,6)$ and $f(x)=x^2 + x + 1$ (among multiples other rational functions),
we get that for of all primitive 
element
$\alpha \in \F_{q^n}$, normal over $\F_q$,
$f(\alpha)$
is not primitive.

For $(q,n)=(3,3)$
%, $a \in \mathbb{F}_{q^n}$
%such that $a^3-a+1=0$
and $f(x)=x^2 + x + 2$ (among multiples other rational functions),
we get that for of all primitive 
element
$\alpha \in \F_{q^n}$, normal over $\F_q$,
$f(\alpha)$
is not primitive.

For $(q,n)=(3,4)$, $a \in \mathbb{F}_{q^n}$
such that $a^4-a^3-1=0$
and $f(x)=\frac{ax + 2a^3 + 2a^2+1}{x + 2a}$ (among multiples other rational functions),
we get that for of all primitive 
element
$\alpha \in \F_{q^n}$, normal over $\F_q$,
$f(\alpha)$
is not primitive.
\end{proof}

The following theorem summarizes the above results.
\begin{theorem}
Let $q$ be a prime power
and let $n \geq 3$.
We have $(q,n) \in \mathcal{B}(3,2)$ except possibly for
$$
\begin{array}{l}
n=3 \quad \text{and} \quad q \neq 8 \text{ and }q \leq 37, \text{ or }
                              q \in \{43,49,61,67,71,79,81,121,151,211,331 \}; \\
n=4 \quad \text{and} \quad q \leq 37 \text{ or }
                              q \in \{41,43,47,83\} ;\\
n=5 \quad \text{and} \quad q \in \{3,4,5,7,9,11,16 \} ;\\
n=6 \quad \text{and} \quad q \in \{2,3,4,5,7, 9,11,13,19 \} ;\\
n=7 \quad \text{and} \quad q \in \{3,4,8\} ;\\
n=8 \quad \text{and} \quad q \in \{2,3,4,5,7,9\} ;\\
q=2 \quad \text{and} \quad n \in \{9,10,12,14,15,16,18,20,24\} ;\\
q=3 \quad \text{and} \quad n \in \{9,10,12,16\} ;\\
q=4 \quad \text{and} \quad n \in \{9,10,12\};  \quad \text{and} \quad
(q,n)  \in  \{(5,12),(7,12)\} .
\end{array}
$$
We also have that $(2,3),(2,4),(2,6),(3,3),(3,4) \notin \Bb(3,2)$.
\end{theorem}

As mentioned after Definition \ref{conjunto-B}, 
if $n_1\leq m_1$ and $n_2 \leq m_2$ then
$\Bb(m_1,m_2) \subset \Bb(n_1,n_2)$. Thus, from the above theorem, we may also 
find information on $\Bb(3,1)$, $\Bb(3,0)$, $\Bb(2,2)$, $\Bb(2,1)$, $\Bb(2,0)$, $\Bb(1,1)$ and $\Bb(1,0)$.

%So, to find the elements which not belong
%to $\Bb(n_1,n_2)$ it suffices to test Lemma \ref{divisores} varying the values 
%of $\ell$
%and $g$,
%take care of Proposition \ref{noBb} for the cases
%$(2,3)$ and $(2,4)$, and
%take care of Proposition \ref{yesBb} for the cases
%$(2,9)$, $(2,15)$ and $(8,7)$,
%where $\phi(2^9-1)=432$, $\mathfrak{N}(2,9)=171$,
%$\phi(2^{15}-1)=27000$, $\mathfrak{N}(2,15)=8430$,
%$\phi(8^7-1)=1778112$ and $\mathfrak{N}(8,7)=698544$.
%In this way, we get the following results:
%
%\begin{theorem}
%Let $q$ be a prime power
%and let $n \geq 3$.
%We have $\Bb(3,2) \subset \mathcal{B}(2,2)$, and
%$(2,9)$, $(8,7)$, $(9, 5)$, $(27, 3)$, $(71, 3)$, $(331, 3) \in \Bb(2,2)$ and
%$(2,3),(2,4) \notin \Bb(2,2)$.
%We have $\Bb(2,2) \subset \mathcal{B}(2,1)$, and
%$(3, 16)$, $(7, 8)$, $(19, 6)$, $(83, 4) \in \Bb(2,1)$,
%and
%$(2,3)$, $(2,4) \notin \Bb(2,1)$.
%We have $\Bb(2,1) \subset \mathcal{B}(1,1)$, and
%$(2,15)$, $(2, 20)$, $(2, 24)$,
%$(3, 9)$, $(16, 4)$, $(17, 3)$, $(25, 4)$, $(37, 4)$,  $(81, 3)$, $(151, 3) 
%\in \Bb(1,1)$
%and
%$(2,3) \notin \Bb(1,1)$.
%\end{theorem}
%
%
%

\section*{Acknowledgements}
C\'{\i}cero Carvalho  was partially funded by FAPEMIG APQ-01645-16, Jo\~{a}o 
Paulo Guardieiro was partially funded by CAPES 88882.441370/2019-01, Victor 
G.L.\ Neumann was 
partially funded by FAPEMIG APQ-03518-18 and Guilherme Tizziotti was 
partially funded by 
CNPq 307037/2019-3.


\begin{thebibliography}{99}

\bibitem{mersenne} Anju and R.K. Sharma, {\em Existence of some special 
primitive normal elements over finite fields}. Finite Fields and Their 
Applications, v. 46, pp. 280--303, 2017.


%\bibitem{carlitz} L. Carlitz, {\em Primitive roots in a finite field}, Transactions of the American Mathematical Society 73 (1952), %373--382.


\bibitem{CGJV} C. Carvalho, J.P. Guardieiro Sousa, V. Neumann and G. 
Tizziotti, {\em On existence of some special pair of primitive elements over 
finite fields}, preprint {\tt arXiv:2002.01867 [math.NT]}. To appear in 
Finite Fields and Their Applications.

%\bibitem{Cohen} S.D.\ Cohen, {\em Pair of primitive elements in fields of even 
%order}, Finite Fields Appl. v. 28, pp. 22--42, 2014.

\bibitem{CH} S.D.\ Cohen and S.\ Huczynska, {\em The primitive normal basis 
theorem -- without a computer}. Journal of London Mathematical Society, v. 67, 
n. 1, pp. 41--56, 2003.

\bibitem{CH2} S. D. Cohen and S. Huczynska, {\em The strong primitive normal basis theorem}. Acta Arith., 143 (4), pp. 299--332, 2010.

\bibitem{CSS}
S.D.\ Cohen. H.\ Sharma and R.\ Sharma, {\em Primitive values of rational 
functions at primitive elements of a finite field},  Journal of Number Theory, 
v. 219, pp. 237--246, 2021.

\bibitem{Fu} L. Fu and D.Q. Wan, {\em A class of incomplete character sums}, 
Quart. J. Math. 65, pp. 1195--1211, 2014.


\bibitem{HBC} H.\ Hazarika, D.K.\ Basnet and S.D.\ Cohen, {\em The existence of 
primitive normal elements of quadratic forms over finite fields}, 
preprint {\tt arXiv:2001.06977 [math.NT]}. To appear in Journal of Algebra and 
Its Applications.


\bibitem{HB} H.\ Hazarika and  D.K.\ Basnet, {\em On existence of primitive 
normal elements of rational form over finite fields of even characteristic},
preprint {\tt arXiv:2005.01216 [math.NT]}.

%\bibitem{HK} H. Hazarika and D. Kumar Basnet, {\em Sufficient condition for 
%existence of special type of primitive normal elements over finite fields}, 
%preprint arxiv.org/abs/1902.04736.

%\bibitem{kapeta0} G. Kapetanakis, {\em An extension of the (strong) primitive 
%normal basis theorem},
% AAECC 25, 311--337 (2014)

\bibitem{kapeta} G. Kapetanakis, Normal bases and primitive elements over 
finite fields, Finite Fields Appl. 26, 
pp. 123--143, 2014.

\bibitem{Kapetanakis-Reis} G. Kapetanakis, L. Reis, {\em Variations of the Primitive
	Normal Basis Theorem}.
Designs, Codes and Cryptography 87 (2019) 1459--1480.

\bibitem{Lenstra} H.W. Lenstra and R.J. Schoof, {\em Primitive Normal Bases for Finite Fields}.
Mathematics of Computation, v. 48, p. 217--231, 1987.

%\bibitem{SAG} K. Sharma, A. Awasthi and A. Gupta, {\em Existence of pair of 
%primitive elements over finite fields of characteristic 2}. Journal of Number 
%Theory, v. 193, p. 386--394, 2018.
%

%\bibitem{SAGE} SageMath, the Sage Mathematics Software System (Version 8.1),
%   The Sage Developers, 2020, \texttt{https://www.sagemath.org}.
%   
   
\bibitem{LN} R. Lidl and H. Niederreiter, {\em Finite Fields}. Cambridge 
university press, 1997.


\end{thebibliography}
\end{document}